\documentclass[a4,11pt]{article}
\usepackage{authblk}
\usepackage[round,authoryear]{natbib} 
\usepackage{amsmath,amsfonts,amssymb,amsthm,color,epsfig,url,bm,bbm,enumerate,dsfont}

\theoremstyle{definition}

\newtheorem{thm}{Theorem}[section]

\newtheorem{lem}{Lemma}[section]

\newtheorem{remark}{{Remark}}[section]

\def\mR{\mathbb{R}}
\def\mC{\mathbb{C}}

\def\tr{\mathrm{tr}}
\def\var{\mbox{var}}

\def\cov{\mbox{cov}}

\def\sign{\mbox{sign}}

\def\diag{\mbox{diag}}

\newcommand{\bSig}{\mbox{\boldmath $\Sigma$}}
\newcommand{\brho}{\bm{\rho}}
\newcommand{\bvare}{\bm{\varepsilon}}

\newcommand{\tbrho}{\widetilde{\bm{\rho}}}

\newcommand{\one}{\mathbf{1}}
\newcommand{\E}{\mathbb{E}}

\newcommand{\F}{\mathcal{F}}
\newcommand{\CC}{\mathcal{C}}

\def\H{{\bf H}}

\def\X{{\bf X}}

\def\g{{\bf g}}

\def\y{{\bf y}}
\def\Z{{\bf Z}}

\def\t{{\bf t}}
\def\R{{\bf R}}
\def\r{{\bf r}}

\def\A{{\bf A}}

\def\G{{\bf G}}

\def\B{{\bf B}}

\def\D{{\bf D}}
\def\M{{\bf M}}
\def\K{{\bf K}}

\def\U{{\bf U}}
\def\V{{\bf V}}
\def\e{{\bf e}}
\def\s{{\bf s}}

\def\bI{{\bf I}}

\newcommand{\trans}{^\top}
\def\defby {\stackrel{\mbox{\textrm{\tiny def}}}{=}}
 
\def\toas{\,{\buildrel a.s. \over \longrightarrow}\,}

\def\tod{\,{\buildrel d \over \longrightarrow}\,}

\begin{document}

\title{Large dimensional Spearman's rank correlation matrices: The central limit theorem and its applications}
\author{Hantao Chen\thanks{htchen2000@sjtu.edu.cn} }
\author{Cheng Wang\thanks{Corresponding author: chengwang@sjtu.edu.cn}}
\affil{School of Mathematical Sciences, MOE-LSC, Shanghai Jiao Tong University, Shanghai 200240, China.}
\date{\today}

\maketitle

\begin{abstract}
    This paper is concerned with Spearman's correlation matrices under large dimensional regime, in which the data dimension diverges to infinity proportionally with the sample size. We establish the central limit theorem for the linear spectral statistics of Spearman's correlation matrices, which extends the results of [\emph{Ann. Statist.} 43(2015) 2588--2623]. We also study the improved Spearman's correlation matrices [\emph{Ann. Math. Statist} 19(1948) 293--325] which is a standard U-statistic of order 3. As applications, we propose three new test statistics for large dimensional independent test and numerical studies demonstrate the applicability of our proposed methods.
\end{abstract}

\section{Introduction}
In multivariate statistical analyses, the covariance matrix is a fundamental tool used to describe the relationships among features. Its theoretical property is crucial for understanding many statistical methods. For the classical setting where the data dimension $p$ is fixed and the sample size $n$ tends to infinity, these properties and their applications in various methods are summarized in textbooks, e.g., \cite{anderson1984introduction}. 

In the last few decades, large amounts of work are focused on the large dimensional regime,
\begin{align}
    n\to\infty,\quad p=p(n)\to\infty,\quad p/n=y_n\to y\in(0,\infty).\label{large_dimensional}
\end{align}
Random matrix theory, as a powerful tool, provides insights into the behavior of large dimensional sample covariance matrices, extending the famous Wishart distribution theory. The pioneering work \cite{marvcenko1967distribution} derived the limiting spectral distribution (LSD) which is called Mar{\u{c}}enko-Pastur (MP) law. With the LSD, we can describe the limits of linear spectral statistics (LSS).  Furthermore, \cite{bai2004clt} firstly derived the central limit theorem (CLT) of LSS. The following works include \cite{pan2008central,anderson2008clt, lytova2009central, pan2014comparison, zheng2015substitution} and so on. As applications of RMT on the sample covariance matrix, \cite{dobriban2018high} and \cite{wang2018dimension} studied the prediction errors of ridge regression and regularized linear discriminant analysis; \cite{hastie2022surprises} demonstrated the double descent phenomenon in the simple linear regression; \cite{bai2009corrections} proposed a bias correction to the likelihood ratio test; \cite{wang2013identity} and \cite{wang2013sphericity} considered the identify test and the sphericity test of  covariance matrices, respectively. For more results on large dimensional covariance matrices, it is referred to \cite{paul2014random} and \cite{yao2015sample} for a comprehensive review.

Normalization is a common procedure in data analysis. By standardizing the sample covariance matrix, we obtain Pearson's correlation matrix, a scale-invariant measure. Recent research has extensively studied Pearson's correlation matrices. \cite{jiang2004limiting} first derived the limiting spectral distribution. \cite{bao2012tracy} and \cite{pillai2012edge} established limiting distributions for the extreme eigenvalues. \cite{mestre2017correlation} and \cite{gao2017high} developed the CLT of LSS of Pearson's correlation matrices. \cite{zheng2019test} extended the CLT to general covariance structures. See also \cite{jiang2019determinant} and \cite{parolya2024logarithmic}. For a large class of population distributions, \cite{el2009concentration} demonstrated that the spectral properties of Pearson's correlation matrices resemble those of sample covariance matrices.  Typically, these studies assume finite fourth moments for the features. However, for distributions with infinite fourth moments, such as heavy-tailed populations, the applicability of these results may require additional verification or may no longer hold. For instance, \cite{2023-AIHP1368} justified the CLT of log-determinant statistics of Pearson's correlation matrices and \cite{heiny2022limiting} discovered a new LSD result for heavy-tailed distributions. 

To address the challenges posed by heavy-tailed data, non-parametric statistics offer robust correlation measures. Among these, Spearman's rank correlation matrix and Kendall's rank correlation matrix are particularly popular due to their distribution-free nature, making them suitable for heavy-tailed data. Recent research has explored the properties of these rank-based correlation matrices. For example, \cite{leung2018testing} and \cite{wang2024rank} studied a class of rank-based U-statistics for independence test. In the realm of random matrix theory, \cite{bai2008large} and \cite{wu2022limiting} investigated the LSD of Spearman's correlation matrices. \cite{bandeira2017marvcenko} and \cite{li2023eigenvalues} studied the LSD of Kendall's correlation matrices. As far as the CLT, \cite{bao2015spectral} considered asymptotic distributions of polynomial functions of Spearman's correlation matrices and \cite{li2021central} studied the CLT of LSS of Kendall's correlation matrices. 

In this paper, we focus on Spearman's correlation matrices and aim to establish a central limit theorem for general linear spectral statistics. Due to introducing ranking, the independence among samples are violated and thus, we turn to consider Gram matrices.  The rescaled Gram matrix is a sample covariance matrix related to the distribution which are independent and uniformly distributed on the permutations of $\{1,\cdots,n\}$. In \cite{bao2015spectral}, they adopted the celebrated moment method and derived the CLT for polynomial functions. In this work, we follow the classical technique developed by \cite{bai2004clt} and consider the asymptotic distribution of Stieltjes transforms. Key challenges arise in computing the covariance of quadratic forms and establishing concentration inequalities for these forms. For uniform distribution on $\{1,\cdots,n\}$, it is challenging to derive the explicit covariance of quadratic forms. We derive the three leading terms which all contribute to the final CLT. More details can be found in our Lemma \ref{lem:second_moment} and Lemma \ref{lem:higher_moments_control}. The obtained results are consistent with \cite{bao2015spectral} for polynomial functions and are also applicable to more general LSS such as log-determinant functions. The resulting CLT of Stieltjes transform can connect to many other covariance or correlation matrices.  

In non-parametric statistics, \cite{Hoeffding1948} theoretical analyzed the Spearman's correlation from the perspective of U-statistics and proposed an improved version. Specifically, Spearman's correlation can be expressed as a U-statistics of order $3$ with an additional term. To address this, \cite{Hoeffding1948} introduced an improved Spearman's correlation which is a standard U-statistic of order $3$. Sample covariance matrices and Kendall's correlation matrices are well-known examples of U-statistics of order $2$, and their CLTs have been extensively studied in \cite{pan2014comparison} and \cite{li2021central}, respectively. To the best of our knowledge, there are no CLTs for general LSS of U-statistics of order higher than $2$. While the improved Spearman's correlation matrix is challenging to analyze directly, we can evaluate the difference between it and the classical Spearman's correlation matrix. This approach enables us to establish a CLT for standard U-statistics of order 3. This result is of interest for covariance/correlation matrices of U-statistic types and may contribute to the development of CLTs for LSS of general U-statistics of higher order.  

As applications of such CLTs, we revisit hypothesis testing for independence. Numerous studies have proposed various test statistics based on different correlation matrices, including \cite{jiang2004asymptotic}, \cite{zhou2007asymptotic}, \cite{gao2017high}, \cite{bao2015spectral}, \cite{leung2018testing}, \cite{bao2019tracy_kendall}, \cite{li2021central}. Our proposed test statistics fall into two categories: those based on Euclidean distance and those based on Stein's loss. Through extensive numerical experiments, we demonstrate the competitive performance of our proposed methods compared to well-established approaches.

Our contributions are summarized as follows:
\begin{enumerate}
	\item For Gram matrices, we study a novel population distribution which is uniformly distributed on the permutations of $\{1,\cdots,n\}$. Unlike the independent component model or elliptical distributions, the quadratic forms associated with this distribution exhibit a complex covariance structure. By carefully analyzing three leading terms, we derive a new central limit theorem.

	\item For Spearman's correlation matrices, we establish a CLT of general linear spectral statistics, extending the work of \cite{bao2015spectral} which focused on polynomial functions. Our approach, based on classical random matrix techniques and the Stieltjes transform, provides a more direct connection to other classical results, shedding light on the underlying structure of Spearman's correlation.
	
    \item From a U-statistic perspective, Spearman's correlation is not a standard U-statistic. \cite{Hoeffding1948} proposed an improved version which is a U-statistic of order 3. By carefully evaluating the difference between the classical and improved Spearman's correlation matrices, we derive the explicit impact on the asymptotic mean and establish a CLT for the improved Spearman's correlation matrix. As we know, this is the first CLT for standard U-statistic of order 3  in random matrix theory. 
    \item  Spearman's correlation matrices, derived from ranking and standardizing the original data matrix, can be viewed as both sample covariance and Pearson-type correlation matrices. From a U-statistic perspective, Spearman's correlation matrices are of order 3, while Kendall's correlation matrices are of order 2.  Thus, Spearman's correlation matrices can be connected with many existing random matrix models and the corresponding CLT results can also be connected with well-established CLT results. The obtained results allow us to gain deeper insights into the asymptotic distribution of linear spectral statistics for various sample covariance and correlation matrices.
\end{enumerate}

The remainder of the paper is structured as follows: Section 2 introduces the necessary background knowledge and tools from random matrix theory. Section 3 presents our main results, including the CLTs for Gram matrices, Spearman's correlation matrices, and improved Spearman's correlation matrices. Section 4 applies our theoretical results to hypothesis testing for independence and conducts numerical experiments to demonstrate the effectiveness of our proposed methods.  In Section 5, we summarize our CLTs with discussions and the Appendix provides detailed proofs of our theoretical results.

\section{Preliminary result in RMT}
Let $\H_n$ be any $n\times n$ Hermitian matrix with eigenvalues $\lambda_1 \geq \cdots \geq \lambda_n$. The empirical spectral distribution (ESD) is defined as 
\begin{align}
    F^{\H_n}(x)=\frac{1}{n}\sum_{i=1}^nI(\lambda_i\leq x), 
\end{align}
where $I(\cdot)$ is the indicator function. If $F^{\H_n}$ converges weakly to some limiting distribution $F$, then we call $F$ the limiting spectral distribution of $\H_n$. 

With the LSD, we can study the linear spectral statistic  which is defined as
\begin{align*}
    \frac{1}{n} \sum_{i=1}^n f(\lambda_i)=\int f(x) dF^{\H_n}(x).
\end{align*}
Here $f(\cdot)$ is any bounded and continuous function. By the property of weak convergence, we can conclude 
\begin{align*}
    \int f(x) dF^{\H_n}(x) \to \int f(x) dF(x).
\end{align*}
Some common functions in statistics include 
\begin{align*}
    \frac{1}{n} \sum_{i=1}^n \lambda^k_i=&\frac{1}{n}\tr(\H^k_n),~k=1,2,\cdots,\\
    \frac{1}{n} \sum_{i=1}^n (\lambda_i-1)^2=&\frac{1}{n} \|\H_n-\bI_n\|_F^2,\\   
    \frac{1}{n} \sum_{i=1}^n \lambda_i-\log(\lambda_i)-1=&\frac{1}{n}\tr(\H_n)-\frac{1}{n}\log\det(\H_n)-1, 
\end{align*}
and so on. If $\H_n$ is a random matrix, we can further consider the central limit theorem of linear spectral statistics.

In random matrix theory, one of the most powerful tools is Stieltjes transform, which is defined as 
\begin{align}
    m_F(z)=\int\frac{1}{x-z}dF(x), \quad z\in\mC^+,
\end{align}
with respect to any distribution function $F$. Here $\mC^+$ is the upper half space of the complex plane. Similar to the characteristic function in probability, there is a one-to-one correspondence between the probability distribution and its Stieltjes transform, e.g.,
\begin{align*}
    F_n \tod F~\iff m_{F_n}(z) \to m_F(z),~\forall~z\in\mC^+.
\end{align*}
With the Stieltjes transform, by the residue theorem of complex analysis,
\begin{align*}
    \frac{1}{n}\sum_{i=1}^n f(\lambda_i)=\int f(x) dF^{\H_n}(x)=-\frac{1}{2\pi i}\oint_{\mathcal{C}} f(z)m_{F^{\H_n}}(z)dz,
\end{align*}
where $\oint_{\mathcal{C}}$ is closed and taken in the positive direction, enclosing the support of $F^{\H_n}$. Furthermore, we can study the asymptotic distribution, e.g.,
\begin{align*}
    \int f(x) dF^{\H_n}(x)-\int f(x) dF(x)  =\frac{1}{2\pi i}\oint_{\mathcal{C}} f(z)\left(m_F(z)-m_{F^{\H_n}}(z)\right)dz.
\end{align*}

In summary, to find the LSD of a random matrix $\H_n$, we can study its Stieltjes transform
\begin{align*}
    m_{F^{\H_n}}(z)=\frac{1}{n}\tr\left(\H_n-z \bI_n\right)^{-1}.
\end{align*}
To explore the asymptotic distribution of the LSS, we need to find the asymptotic distribution of 
\begin{align*}
    m_F(z)-m_{F^{\H_n}}(z),  
\end{align*}
which is usually a Gaussian process. The Gaussian process further yields the asymptotically normal distribution of the LSS. It is referred to \cite{bai2010spectral} for a comprehensive survey on random matrix theory. 

\section{Main result}
For independent and identically distributed (i.i.d.) samples $\X_1,\ldots,\X_n\in\mR^p$,   
we denote their rank statistics as $\r_k=(r_{i1},\cdots,r_{ip})\trans,~i=1,\ldots,n$. For each feature $j \in \{1,\cdots,p\}$, $(r_{1j},\ldots,r_{nj})$ are  uniformly distributed on the permutations of $\{1,\cdots,n\}$. Then, 
\begin{align*}
    \E r_{ij}=\frac{n+1}{2},~\var(r_{ij})=\frac{n^2-1}{12}.
\end{align*}
With the rank statistics, the Spearman's rank correlation matrix is defined by
\begin{align}\label{form:spearman's_rho}
    \brho_n=\frac{12}{n(n^2-1)}\sum_{k=1}^{n}(\r_k-\frac{n+1}{2}\one_p)(\r_k-\frac{n+1}{2}\one_p)\trans,
\end{align}
which is the Pearson's correlation matrix based on $\r_1,\ldots,\r_n \in \mR^p$. Due to ranking, $\r_1,\ldots,\r_n \in \mR^p$ are not independent anymore and it is hard to tackle the Spearman's rank correlation matrix directly. Here we turn to its Gram matrix.

\subsection{Gram matrix}
Standardizing rank statistics, we denote
\begin{align*}
\sqrt{\frac{12}{n^2-1}}\begin{pmatrix}
r_{11}-\frac{n+1}{2}&\cdots&r_{1p}-\frac{n+1}{2}\\
\vdots&\cdots&\vdots\\
r_{n1}-\frac{n+1}{2}&\cdots&r_{np}-\frac{n+1}{2}
\end{pmatrix}\defby  \begin{pmatrix}
    \s_1,&\cdots,&\s_p
\end{pmatrix}.
\end{align*}
If the features are completely independent, $\s_1,\ldots,\s_p \in \mR^n$ are i.i.d. 
and have been centered, e.g.,
\begin{align*}
    \E \s_i=\mathbf{0}_n,~\bSig\defby  \cov(\s_1)=\frac{n}{n-1}\left(\bI_n-\frac{1}{n}\one_n\one_n\trans\right).
\end{align*}
Then, we can study the sample covariance matrix of $\s_1,\ldots,\s_p$, 
\begin{align}\label{form:gram-type}
    \g_n=\frac{1}{p}\sum_{i=1}^p\s_i\s_i\trans.
\end{align}
This sample covariance matrix can also be regarded as the Gram matrix of the original rank statistics, that is,
\begin{align*}
     \g_n=\frac{12}{p(n^2-1)}\left( (\r_k-\frac{n+1}{2}\one_p)\trans(\r_k-\frac{n+1}{2}\one_p)\right)_{n \times n}.
\end{align*}
Thus, $\g_n$ and $\brho_{n}/y_n$ share the same non-zero eigenvalues.

Denoting $m_n(z)$ as the Stieltjes transforms of $\brho_{n}$ i.e.,
\begin{align*}
    m_n(z)=\frac{1}{p}\tr(\brho_n-z \bI_p)^{-1},
\end{align*}
it is proven in \cite{bai2008large} that $m_n(z) \to m(z)$ almost surely and 
\begin{align}
m(z)=\frac{1-y-z+\sqrt{(1+y-z)^2-4y}}{2yz}.
\end{align}
This result shows that the LSD of $\brho_n$ converges weakly to the M-P law $F_{y}$ almost surely, whose density function is
\begin{align*}
    p_y(x)=\frac{\sqrt{(x-(1-\sqrt{y})^2)((1+\sqrt{y})^2-x)}}{2\pi xy}I\left((1-\sqrt{y})^2<x<(1+\sqrt{y})^2\right),
\end{align*}
for $y\leq 1$ and has a point mass $1-1/y$ at origin for $y>1$.

We further denote $s_n(z)$ and $\underline{s}_n(z)$ as the Stieltjes transform of $\g_n$ and $\brho_n/y_n$, respectively
\begin{align*}
    &s_n(z)=\frac{1}{n}\tr(\g_n-z \bI_n)^{-1}=y_n^2 m_n(y_n z)-\frac{1-y_n}{z},\\
    &\underline{s}_n(z)=\frac{1}{p}\tr(\frac{\brho_n}{y_n}-z\bI_p)^{-1}=\frac{1}{y_n}(s_n(z)+\frac{1}{z})-\frac{1}{z},
\end{align*}
and almost surely 
\begin{align*}
    s_n(z) \to s(z)
    =&\frac{1-y_0-z+\sqrt{(1+y_0-z)^2-4y_0}}{2y_0z},\\
    \underline{s}_n(z)\to\underline{s}(z)=&\frac{-(1-y_0+z)+\sqrt{(1+y_0-z)^2-4y_0}}{2z},
\end{align*}
where $y_0=1/y$. Then, the LSD of $\g_n$ converges weakly to the M-P law $F_{y_0}$ almost surely.

For the LSS of $\g_{n}$,
\begin{align*}
    \int f(x)dF^{\g_{n}}(x)=\frac{1}{n}\sum_{i=1}^nf\left(\lambda_i(\g_{n})\right),
\end{align*}
where $f$ is an analytic function and $\lambda_1(\g_{n}) \geq\cdots \geq \lambda_n(\g_{n})$ are eigenvalues of $\g_n$, we have almost surely 
\begin{align*}
   \int f(x)dF^{\g_{n}}(x) \to  \int f(x)dF_{y_0}(x).
\end{align*}
Further, we study the asymptotic distribution of the LSS. Let 
\begin{align*}
    G_n(x)=n\left(F^{\g_{n}}(x)-F_{n/p}(x)\right),
\end{align*}
and we focus on 
\begin{align}
    \int f(x)dG_n(x)=n\left(\int f(x)dF^{\g_{n}}(x)-\int f(x)dF_{n/p}(x)\right)\label{form:deviation}.
\end{align}
Our central limit theorem is presented as follows.
\begin{thm}\label{thm:LSS_clt}
Assume that $\{X_{ij}: i=1,\ldots,n; j=1,\ldots,p\}$ are doubly independent and absolutely continuous with respect to the Lebesgue measure. Let $f_1,\ldots,f_k$ be functions on $\mR$ and analytic on an open interval containing 
    \begin{align}
        [I(y_0<1)(1-\sqrt{y_0})^2,(1+\sqrt{y_0})^2].\label{form:support}
    \end{align}
Then, as $n/p \to y_0 \in (0,\infty)$, the random vector 
    \begin{align*}
        \left(\int f_1(x)dG_n(x),\cdots,\int f_k(x)dG_n(x)\right)
    \end{align*}
converges weakly to a Gaussian vector $(G_{f_1},\cdots,G_{f_k})$ with the asymptotic mean 
    \begin{align*}
        \E G_{f}=-\frac{1}{2\pi i}\oint_{\mathcal{C}} f(z)\mu(z)dz,
    \end{align*}
and the asymptotic covariance function 
    \begin{align*}
        \cov(G_{f},G_{g})=-\frac{1}{4\pi^2}\oint_{\mathcal{C}_1} \oint_{\mathcal{C}_2}f(z_1)g(z_2)\sigma(z_1,z_2)dz_1dz_2,
    \end{align*}
    where 
    \begin{align*}
        \mu(z)=&\frac{y_0\underline{s}^3(z)\left(1+\underline{s}(z)\right)}{\left(\left(1+\underline{s}(z)\right)^2-y_0\underline{s}^2(z)\right)^2}-\frac{2y_0\underline{s}^3(z)}{\left(\left(1+\underline{s}(z)\right)^2-y_0\underline{s}^2(z)\right)\left(1+\underline{s}(z)\right)}\\
        &+\frac{\underline{s}^3(z)}{\left(1+\underline{s}(z)\right)^2-y_0\underline{s}^2(z)},\\
    \sigma(z_1,z_2)=&\frac{2\underline{s}'(z_1)\underline{s}'(z_2)}{\left(\underline{s}(z_1)-\underline{s}(z_2)\right)^2}-\frac{2}{(z_1-z_2)^2}-\frac{2y_0\underline{s}'(z_1)\underline{s}'(z_2)}{\left(1+\underline{s}(z_1)\right)^2\left(1+\underline{s}(z_2)\right)^2}.
    \end{align*}
    The contour $\oint_{\mathcal{C}}$ is closed and taken in the positive direction, each enclosing the support \eqref{form:support}.
\end{thm}
For concrete functions such as logarithms and polynomials, we will derive CLTs in next section. In details, the integral involving $\underline{s}(z)$ can be calculated explicitly for most cases. Furthermore, we present a remark on the asymptotic mean and covariance.
\begin{remark} \label{remark1}
Regarding $\s_1,\ldots,\s_p$ as an i.i.d. sample from the population distribution $\s \in \mR^n$ with 
\begin{align*}
    \E \s=\mathbf{0}_n,~\cov(\s)=\bSig=\frac{n}{n-1}\left(\bI_n-\frac{1}{n}\one_n\one_n\trans\right),
\end{align*}
we have
\begin{align*} 
    \cov(\s \trans\A\s,\s \trans\B\s)=2\tr(\A\B)-\frac{6}{5}\tr(\A\circ\B)-\frac{4}{5n}\tr(\A)\tr(\B)+O\left(\|\A\|\|\B\|\right),
\end{align*}
where $\A$ and $\B$ are non-random $n\times n$ symmetric matrices. Denoting 
\begin{align*} 
\mu(z)\defby  \mu_1(z)+\mu_2(z)+\mu_3(z),~\sigma(z_1,z_2)\defby  \sigma_1(z_1,z_2)+\sigma_2(z_1,z_2)+\sigma_3(z_1,z_2),
\end{align*}
$\mu_1(z)$ and $\sigma_1(z_1,z_2)+\sigma_2(z_1,z_2)$ firstly appeared in the seminal work of \cite{bai2004clt} which are due to the term $2\tr(\A\B)$.  
The second term $\tr(\A\circ\B)$ was firstly studied by \cite{pan2008central}. The third term $\tr(\A)\tr(\B)/n$ is new which has similar effect as $\tr(\A\circ\B)$ under the current setting and these two terms contributes to $\mu_2(z)$ and $\sigma_3(z_1,z_2)$ together. The remaining $\mu_3(z)$ comes from the difference between $\bSig$ and $\bI_n$. More specifically, the Mar\u{c}enko-Pastur equation of $F_{n/p}$ is
\begin{align*}
    s_n^{(0)}(z)=\frac{1}{(1-\frac{n}{p}-\frac{n}{p} z s_n^{(0)}(z))-z}.
\end{align*}
For $\bSig$ which has $n-1$ eigenvalues equal to $n/(n-1)$ and one zero eigenvalue, the corresponding Mar\u{c}enko-Pastur equation is 
\begin{align*}
    s_n^{(1)}(z)=&\int\frac{1}{t(1-\frac{n}{p}-\frac{n}{p}z s_n^{(1)}(z))-z}dF^{\bSig}(t)\\
    =&\frac{n-1}{n}\left[\frac{n}{n-1}\left(1-\frac{n}{p}-\frac{n}{p}zs\right)-z\right]^{-1}-\frac{1}{nz}.
\end{align*}
Through careful calculation, we can conclude 
\begin{align*}
    n \left(s_n^{(1)}(z)- s_n^{(0)}(z)\right) \to \frac{\underline{s}^3(z)}{\left(1+\underline{s}(z)\right)^2-y_0\underline{s}^2(z)} \defby  \mu_3(z).
\end{align*}
More details can be found in the proof, e.g., the equation \eqref{form:expect_other_hand}.
\end{remark}

\subsection{Spearman's rank correlation matrix}
For Spearman's rank correlation matrix $\brho_n \in \mR^{p \times p}$ which has the same non-zero eigenvalues as the ones of $y_n \g_n \in \mR^{n \times n}$, we have
\begin{align*}
    \int f(x)dF^{\brho_n}(x)=&\frac{1}{p}\sum_{i=1}^{p}f(\lambda_i(\brho_n))=\frac{1}{p}\sum_{i=1}^{n}f(y_n\lambda_i(\g_n))+\frac{p-n}{p}f(0)\\
    =& \frac{1}{y_n } \int f(y_n x) d F^{\g_n}(x)+\frac{p-n}{p}f(0).
\end{align*} 
In addition, by the property of M-P law,
\begin{align*}
    \int f(y_nx)dF_{n/p}(x)=y_n\int f(x)dF_{y_n}(x)+(1-y_n)f(0),
\end{align*}
which yields 
\begin{align*}
    \int f(x)dF_{y_n}(x)= \frac{1}{y_n }\int f(y_nx)dF_{n/p}(x)+\frac{p-n}{p}f(0).
\end{align*}
Therefore,  we can study the asymptotic distribution of 
\begin{align*}
    T(f)\defby  p\left(\int f(x)dF^{\brho_n}(x)-\int f(x)dF_{y_n}(x)\right)=\int f(y_nx)dG_n(x).
\end{align*}

By Theorem \ref{thm:LSS_clt}, we have proven the CLT of $\int f(yx)dG_n(x)$ and the remaining is to show
\begin{align*}
    \int f(y_nx)dG_n(x)- \int f(yx)dG_n(x) =o_p(1),
\end{align*} 
whose details can be found in the proof. Based on these observations, we state the CLT for $\brho_n$ in the following theorem.
\begin{thm}\label{thm:LSS_clt_2}
Assume that $\{X_{ij}: i=1,\ldots,n; j=1,\ldots,p\}$ are doubly independent and absolutely continuous with respect to the Lebesgue measure.  Let $f_1,\cdots,f_k$ be functions analytic on an open interval containing
    \begin{align}
        \left[I(y<1)(1-\sqrt{y})^2,(1+\sqrt{y})^2\right].\label{form:support_2}
    \end{align}
    Then, as $p/n \to y \in (0,\infty)$, the random vector $(T(f_1),\cdots,T(f_k))$ converges weakly to a Gaussian vector $(Z_{f_1},\cdots,Z_{f_k})$ with mean function 
    \begin{align}
        \E Z_{f}=-\frac{1}{2\pi i}\oint_{\mathcal{C}'} f(yz)\mu(z)dz=-\frac{1}{2\pi i}\oint_{\mathcal{C}} f(z)\frac{\mu(z/y)}{y}dz,\label{form:mean_function_2}
    \end{align}
    and covariance function 
    \begin{align}
        \cov(Z_{f},Z_{g})=-\frac{1}{4\pi^2}\oint_{\mathcal{C}'_1}\oint_{\mathcal{C}'_2} f(yz_1)g(yz_2)\sigma(z_1,z_2)dz_1dz_2, \nonumber\\
        =-\frac{1}{4\pi^2}\oint_{\mathcal{C}_1}\oint_{\mathcal{C}_2} f(z_1)g(z_2)\frac{\sigma(z_1/y,z_2/y)}{y^2}dz_1dz_2,
        \label{form:cov_function_2}
    \end{align}
    where $\mu(z)$ and $\sigma(z_1,z_2)$ are defined in Theorem \ref{thm:LSS_clt}, and the contour $\oint_{\mathcal{C}}$ is closed and taken in the positive direction, each enclosing the support \eqref{form:support_2}.
\end{thm}

\begin{remark}
    In \cite{bao2015spectral}, they derived the asymptotic distribution of $\tr(\brho_n^k)$ for for any positive integer $k \geq 2$. Technically, they utilized Anderson and Zeitouni's cumulant method \citep{anderson2008clt} and  proposed a two-step comparison approach to obtain the
    explicit mean and covariance of CLTs. Here we adopt the classical proof technique from the seminal work of \cite{bai2004clt}. Taking $f(x)=x^k$, we will revisit the results of \cite{bao2015spectral} in \eqref{form:tr_1}.
\end{remark}

\begin{remark}
   Note that $\brho_n$ is a correlation matrix which means $\tr(\brho_n)=p$. Thus, $f(x)=x$ is a degenerate case.  Taking $f(x)=x$, we can show that the asymptotic mean and the variance are both zero.  
\end{remark}

\begin{remark}
   For sample covariance matrices, using the sample mean or the true population mean has impacts on the final CLT and \cite{pan2014comparison} compared the two types of sample covariance matrices. More specifically,  \cite{zheng2015substitution} proposed a substitution principle which adjusted the sample size from $n$ to $n-1$ for the sample covariance matrix based on the sample mean. For Spearman's rank correlation matrices, although the sample mean of the rank statistics is constant, e.g., 
   \begin{align*}
    \frac{1}{n} \sum_{i=1}^n \r_i=\frac{n+1}{2} \one_p,
   \end{align*}
 it still contributes to the CLT. In details, the population covariance matrix $\bSig \in \mR^{n \times n}$ related the Gram matrix $\g_n$ has one zero eigenvalue. Following \cite{zheng2015substitution}, we can also use $n-1$ and consider the ratio $p/(n-1)$. Then, the additional term such as $\mu_3(z)$ can be removed and more details can be found in Remark \ref{remark1}.     
\end{remark}

\subsection{Improved Spearman's correlation matrix}
The Spearman's rank correlation is a classical method in non-parametric statistics. 
We note that the rank statistics can be transformed into the sum of indicators, which implies 
\begin{align*}
     r_{ij}-\frac{n+1}{2}=\frac{1}{2}\sum_{k\not=i}\sign(X_{ij}-X_{kj}).
\end{align*}
Here $\sign(\cdot)$ is the sign function. Denoting the sign vector 
\begin{align*}
    \A_{ij}=\sign(\X_i-\X_j) \defby \begin{pmatrix}
        \sign(X_{i1}-X_{j1})\\
        \vdots\\
        \sign(X_{ip}-X_{jp})
    \end{pmatrix},  
\end{align*}
we have
\begin{align*}
    \brho_n=\frac{3}{n(n^2-1)}\sum_{i=1}^{n}\sum_{j,k\not=i}\A_{ij}\A_{ik}\trans.
\end{align*}
In the form of non-parametric U-statistics, it can be decomposed into two U-statistics,
\begin{align*}
    \brho_n=\frac{3}{n(n^2-1)}\sum_{i,j}^*\A_{ij}\A_{ij}\trans+\frac{3}{n(n^2-1)}\sum_{i,j,k}^*\A_{ij}\A_{ik}\trans=\frac{3}{n+1}\K_n+\frac{n-2}{n+1}\tbrho_n,
\end{align*}
where 
\begin{align*}
    \K_n=\frac{1}{n(n-1)} \sum_{i,j}^*\A_{ij}\A_{ij}\trans
\end{align*}
is Kendall's rank correlation matrix and
\begin{align*}
    \tbrho_n=\frac{3}{n(n-1)(n-2)}\sum_{i,j,k}^*\A_{ij}\A_{ik}\trans
\end{align*}
is the improved Spearman's rank correlation matrix proposed by \cite{Hoeffding1948}. Here $\sum^*$ denotes summation over mutually different indices and more details can be found in \cite{wu2022limiting}.

As can be seen, the Kendall's correlation matrix $\K_n$ is a U-statistic of order 2 and \cite{li2021central} studied the asymptotic distribution of its linear spectral statistics. The improved Spearman's rank correlation matrix $\tbrho_n$ is a U-statistic of order $3$, that is difficult to analyze it directly. Based on the CLT of $\brho_n$, we can study the difference between $\tbrho_n$ and $\brho_n$ which is $3(\K_n-\tbrho_n)/(n+1)$. For LSD \citep{wu2022limiting}, this difference can be ignored. However, for the CLT of LSS, this deviation does contribute a non-trivial term to the asymptotic distribution. Considering the centered statistic 
\begin{align*}
    \widetilde{T}(f)\defby p\left(\int f(x)dF^{\tbrho_n}(x)-\int f(x)dF_{y_n}(x)\right),
\end{align*}
we present the main result in the following theorem.

\begin{thm}\label{thm:LSS_clt_3}
    Under the conditions of Theorem \ref{thm:LSS_clt_2}, as $y_n\to y$, the random vector $(\widetilde{T}(f_1),\cdots,\widetilde{T}(f_k))$ converges weakly to a Gaussian vector $(\widetilde{Z}_{f_1},\cdots,\widetilde{Z}_{f_k})$ with mean function 
    \begin{align}
        \E \widetilde{Z}_{f}=-\frac{1}{2\pi i}\oint_{\mathcal{C}} f(yz)\left(\mu(z)+\widetilde{\mu}(z)\right)dz,\label{form:mean_function_3}
    \end{align}
    and covariance function 
        \begin{align}
        \cov(\widetilde{Z}_{f},\widetilde{Z}_{g})=-\frac{1}{4\pi^2}\oint_{\mathcal{C}_1} \oint_{\mathcal{C}_2} f(yz_1)g(yz_2)\sigma(z_1,z_2)dz_1dz_2,\label{form:cov_function_3}
    \end{align}
    where $\mu(z)$, $\sigma(z_1,z_2)$ are defined in Theorem \ref{thm:LSS_clt}, and 
    \begin{align*}
        \widetilde{\mu}(z)=\frac{\underline{s}^3(z)\left(2+\underline{s}(z)\right)}{\left(1+\underline{s}(z)\right)^2-\underline{s}^2(z)/y}.
    \end{align*}
\end{thm}

Compared with Theorem \ref{thm:LSS_clt_2} for $\brho_n$, the asymptotic variance is the same and there is a new additional term to the asymptotic mean which is due to the difference $\tbrho_n-\brho_n$. 

\section{Application}
For sample correlation matrices, one important application is to test mutual independence among features. More specifically, we consider the hypotheses testing problem
\begin{align*}
    H_0:  \R=\bI,
\end{align*}
where $\R$ is a population correlation matrix such as the classical Pearson's correlation matirx, Spearman's correlation matrix, Kendall's correlation matrix and so on. 

To evaluate $\R=\bI$, we have the following equivalent definitions
\begin{align*}
   \ell_2~\mbox{loss:} \quad & \|\R-\bI\|_2^2=\tr(\R^2)-2\tr(\R)+p=0;\\
   \mbox{Stein's loss:} \quad &  \tr(\R)-\log(\R)-p=0;\\
   \ell_\infty~\mbox{loss:} \quad & \|\R-\bI\|_\infty=0.
\end{align*}
Based on an estimator $\widehat{\R}$, intuitively one can conduct test statistics
\begin{align*}
T_1(\widehat{\R})=&\tr(\widehat{\R}^2)-2\tr(\widehat{\R})+p,\\
T_2(\widehat{\R})=&\tr(\widehat{\R})-\log(\widehat{\R})-p,\\
T_3(\widehat{\R})=&\|\widehat{\R}-\bI\|_\infty.
\end{align*}
For large dimensional data, it is challenging to derive the asymptotic distribution of these statistics. In the past 20 years, significant progress has been made in this field and many important methods were proposed in literature. In special, $\ell_2$ loss and Stein's loss can be expressed by linear spectral statistics and in RMT, they motivate the study on the LSS of these correlation matrices. We summarize the developments of testing correlation matrices in Table \ref{table1}.  
\begin{table}[ht!]
	\centering
	\caption{Developments of testing correlation matrices in RMT}
	\label{table1}
	\resizebox{1\textwidth}{!}{%
	\begin{tabular}{|c|ccc|}
\hline
& \mbox{Sample~correlation}&\mbox{Kendall's $\tau$}&	\mbox{Spearman's $\rho$}\\
\hline
\mbox{$\ell_2$~loss} & \cite{gao2017high, zheng2019test}& \cite{li2021central}&\cite{bao2015spectral}\\
\mbox{Stein's loss} &\cite{gao2017high}&\cite{li2021central} & \\
\mbox{$\ell_\infty$~loss:}&\cite{zhou2007asymptotic}&\cite{han2017distribution}& \cite{han2017distribution}\\
\hline
\end{tabular}}
\end{table}

In special, \cite{chen2010tests} proposed a U-statistic's trick to estimate the $\ell_2$~loss directly which can obtain a better convergence rate. Using this trick, several statistics based on $T_1(\widehat{\R})$ can be extended to high-dimensional data case where $p \gg n$. See \cite{leung2018testing} for more details.

As application of Theorem \ref{thm:LSS_clt_2} for Spearman's correlation matrix $\brho_n$, we can fill the gap for Stein's loss, i.e., obtaining the asymptotic distribution of $T_2(\brho_n)$. Similarly, based on the improved Spearman's correlation matrix $\tbrho_n$, we can also conduct three test statistics. \cite{han2017distribution} has studied $T_3(\tbrho_n)$ and here we can derive the distributions of $T_1(\tbrho_n)$ and $T_2(\tbrho_n)$.    

Taking $f(x)$ being logarithm $\log(x)$ or polynomial $x^k$ for $k \geq 2$,  by direct calculations of Theorem \ref{thm:LSS_clt_2} and Theorem \ref{thm:LSS_clt_3}, we obtain the following asymptotic distributions.
\begin{thm}\label{thm:calculation}
    Under the conditions of Theorem \ref{thm:LSS_clt_2}, we have
    \begin{align}
        &\log|\brho_n|+(n-p)\log(1-y_n)+p\tod N(\mu_{\log},\sigma_{\log}^2),\label{form:log_1}\\
        &\log|\tbrho_n|+(n-p)\log(1-y_n)+p\tod N(\widetilde{\mu}_{\log},\sigma_{\log}^2),\label{form:log_2}\\
        &\tr(\brho_n^k)-\sum_{j=0}^{k-1}\frac{py_n^j}{(j+1)}\binom{k}{j}\binom{k-1}{j}\tod N(\mu_{x^k},\sigma_{x^k}^2),\label{form:tr_1}\\
        &\tr(\tbrho_n^k)-\sum_{j=0}^{k-1}\frac{py_n^j}{(j+1)}\binom{k}{j}\binom{k-1}{j}\tod N(\widetilde{\mu}_{x^k},\sigma_{x^k}^2),\label{form:tr_2}
    \end{align}
    where the asymptotic means are
    \begin{align*}
        \mu_{\log}=&\frac{3}{2}\log(1-y)+2y,\\
        \widetilde{\mu}_{\log}=&\mu_{\log}-\frac{y^2}{1-y},\\
        \mu_{x^k}=&\frac{1}{4}\left[(1-\sqrt{y})^{2k}+(1+\sqrt{y})^{2k}\right]-\frac{1}{2}\sum_{j=0}^k\binom{k}{j}^2y^{k-j}-\frac{2}{y}\sum_{j=0}^k\binom{k}{j}(y-1)^j\binom{2k-j}{k-2}\\
    &+\sum_{j=0}^k\binom{k}{j}(y-1)^j\binom{2k-j-1}{k-2},\\
    \widetilde{\mu}_{x^k}=&\mu_{x^k}-\sum_{j=0}^{k-1}\binom{k}{j}(y-1)^j\binom{2k-j-2}{k-1}+\sum_{j=0}^k\binom{k}{j}(y-1)^j\binom{2k-j}{k-1},
    \end{align*}
and the asymptotic variances are
\begin{align*}
    \sigma_{\log}^2=&-2\log(1-y)-2y,\\
        \sigma_{x^k}^2=&2\sum_{j_1=0}^{k-1}\sum_{j_2=0}^{k}\binom{k}{j_1}\binom{k}{j_2}\left(y-1\right)^{j_1+j_2}\sum_{l=1}^{k-j_1}l\binom{2k-1-(j_1+l)}{k-1}\binom{2k-1-j_2+l}{k-1}\\
    &-\frac{2}{y}\sum_{j_1=0}^{k}\sum_{j_2=0}^{k}\binom{k}{j_1}\binom{k}{j_2}\left(y-1\right)^{j_1+j_2}\binom{2k-j_1}{k-1}\binom{2k-j_2}{k-1}.
    \end{align*}
\end{thm}
For polynomials of $\brho_n$,  our results \eqref{form:tr_1} are consistent with ones of \cite{bao2015spectral} and the other three asymptotic distributions are new which can be used to derive the asymptotic distribution of test statistics.

Noting $\tr(\brho_n)=p$ and $\tr(\tbrho_n)=p$, we can simplify the test statistics of $T_1(\brho_n),~T_1(\tbrho_n),~T_2(\brho_n),~T_2(\tbrho_n)$ and consider
\begin{align*}
    L_{\brho,2}=\tr(\brho_n^2),\quad L_{\brho,\log}=\log(|\brho_n|),\quad L_{\tbrho,2}=\tr(\tbrho_n^2),\quad L_{\tbrho,\log}=\log(|\tbrho_n|).
\end{align*}
With Theorem \ref{thm:calculation}, we can get four rejection regions for testing the null distribution
\begin{align*}
    R_1 =& \{L_{\brho,2}-\frac{p^2}{n}-p>y_n^2-y_n+2y_nZ_{\alpha}\},\\
    R_2 =& \{L_{\brho,\log}+(n-p)\log(1-y_n)+p<\frac{3}{2}\log(1-y_n)+2y_n\\
    &-\sqrt{-2\log(1-y_n)-2y_n}Z_{\alpha}\},\\
    R_3 =& \{L_{\tbrho,2}-\frac{p^2}{n}-p>3y_n^2-y_n+2y_nZ_{\alpha}\},\\
    R_4 =& \{L_{\tbrho,\log}+(n-p)\log(1-y_n)+p<\frac{3}{2}\log(1-y_n)+\frac{2y_n-3y_n^2}{1-y_n}\\
    &-\sqrt{-2\log(1-y_n)-2y_n}Z_{\alpha}\},
\end{align*}
where $Z_{\alpha}$ is the upper-$\alpha$ quantile of $N(0,1)$.


To examine the finite sample performance of these test statistics, we conduct the following null hypotheses with data $\X_n=(X_{ij})_{n\times p}$ generated from different models. Specifically, we consider three types of null distributions:
\begin{itemize}
    \item Normal distribution: $X_{ij}$ are i.i.d. $N(0,1)$ for $1\leq i\leq n$ and $1\leq j\leq p$.
    \item Cauchy distribution: $X_{ij}$ are i.i.d. Cauchy distribution with location $0$ and scale $1$ (Cauchy$(0,1)$) for $1\leq i\leq n$ and $1\leq j\leq p$.
    \item Mixed distribution: $X_{ij_1}$ are i.i.d. Cauchy$(0,1)$ for $1\leq i\leq n$, $1\leq j_1\leq \lfloor p/4\rfloor$; $X_{ij_2}$ are i.i.d. $N(0,1)$ for $1\leq i\leq n$, $\lfloor p/4\rfloor+1\leq j_2\leq \lfloor p/2\rfloor$; $X_{ij_3}$ are i.i.d. $\chi^2(2)$ for $1\leq i\leq n$, $\lfloor p/2\rfloor+1\leq j_2\leq p$.
\end{itemize}
It is noted that Cauchy$(0,1)$ is a well known heavy-tailed distribution without expectation, and the mixed distribution is from \cite{li2021central}. 

As for comparison, we consider other 8 test statistics based on Spearman, Kendall and Pearson's correlation matrices:
\begin{enumerate}
    \item $L_{\brho,\max}$: maximum test based on Spearman's correlations \citep{han2017distribution};
    \item $L_{\tbrho,\max}$: maximum test based on improved Spearman's correlations \citep{han2017distribution};
    \item $L_{\K,2}$: $\ell_2$ test based on Kendall's correlations \citep{li2021central};
    \item $L_{\K,\log}$: Stein's test based on Kendall's correlations \citep{li2021central};
    \item $L_{\K,\max}$: maximum test based on Kendall's correlations \citep{han2017distribution};
    \item $L_{\R,2}$: $\ell_2$ test based on Pearson's correlations \citep{han2017distribution};
    \item $L_{\R,\log}$: Stein's test based on Pearson's correlations \citep{han2017distribution};
    \item $L_{\R,\max}$: maximum test based on Pearson's correlations \citep{zhou2007asymptotic}.
\end{enumerate}

We conduct numerical experiments to evaluate the performance of our proposed test statistics. We consider various combinations of sample size $n$, dimension $p$, and underlying distributions, and compare our methods with existing approaches. Table \ref{table:emp_size} presents the empirical sizes of the tests at a nominal significance level of 5\% based on 1000 replications. Our results demonstrate that Pearson's correlation-based tests are sensitive to distributional assumptions and may not perform well under heavy-tailed distributions. In contrast, rank-based test statistics, including our proposed $L_{\brho,\log}$, $L_{\tbrho,2}$, and $L_{\tbrho,\log}$, exhibit robust performance across different distributions. The empirical sizes of our proposed tests are close to the nominal 5\% level, confirming the validity of our theoretical results.

\begin{table}[htbp!]
    \centering  
    \caption{Empirical sizes of independence test statistics based on Pearson, Spearman and Kendall's correlations.}
    \label{table:emp_size}
    \resizebox{\textwidth}{!}{
  \begin{tabular}{@{\extracolsep{5pt}} cccccccccc} 
  \\[-1.8ex]\hline 
  \hline \\[-1.8ex]
  n & 100 & 200 & 400 & 100 & 200 & 400 & 100 & 200 & 400 \\
  \hline \\[-1.8ex] 
  p & 50 & 100 & 200 & 70 & 140 & 280 & 200 & 400 & 800 \\
  \hline \\[-1.8ex]
  y & 0.5 & 0.5 & 0.5 & 0.7 & 0.7 & 0.7 & 2 & 2 & 2 \\
  \hline \\[-1.8ex] 
  \multicolumn{10}{c}{Normal distribution}\\
  $L_{\brho,2}$ & 0.046 & 0.05 & 0.046 & 0.044 & 0.041 & 0.045 & 0.051 & 0.059 & 0.058 \\ 
  $L_{\brho,\log}$ & 0.037 & 0.049 & 0.041 & 0.052 & 0.043 & 0.048 & - & - & - \\ 
  $L_{\brho,\max}$ & 0.023 & 0.028 & 0.037 & 0.025 & 0.033 & 0.041 & 0.022 & 0.023 & 0.049 \\ 
  $L_{\tbrho,2}$ & 0.047 & 0.052 & 0.047 & 0.047 & 0.042 & 0.046 & 0.063 & 0.062 & 0.059 \\ 
  $L_{\tbrho,\log}$ & 0.102 & 0.075 & 0.05 & 0.069 & 0.049 & 0.052 & - & - & - \\ 
  $L_{\tbrho,\max}$ & 0.023 & 0.027 & 0.037 & 0.025 & 0.033 & 0.042 & 0.022 & 0.023 & 0.049 \\ 
  $L_{\K,2}$ & 0.043 & 0.05 & 0.047 & 0.046 & 0.045 & 0.047 & 0.059 & 0.067 & 0.06 \\ 
  $L_{\K,\log}$ & 0.047 & 0.054 & 0.045 & 0.055 & 0.045 & 0.046 & 0.119 & 0.09 & 0.064 \\ 
  $L_{\K,\max}$ & 0.039 & 0.039 & 0.048 & 0.035 & 0.042 & 0.048 & 0.034 & 0.029 & 0.058 \\ 
  $L_{\R,2}$ & 0.083 & 0.109 & 0.086 & 0.074 & 0.076 & 0.083 & 0.28 & 0.292 & 0.252 \\ 
  $L_{\R,\log}$ & 0.04 & 0.045 & 0.044 & 0.04 & 0.05 & 0.056 & - & - & - \\ 
  $L_{\R,\max}$ & 0.027 & 0.023 & 0.04 & 0.027 & 0.028 & 0.037 & 0.012 & 0.029 & 0.039 \\ 
  \hline \\[-1.8ex] 
  \multicolumn{10}{c}{Cauchy distribution}\\
  $L_{\brho,2}$ & 0.056 & 0.055 & 0.039 & 0.042 & 0.058 & 0.041 & 0.048 & 0.051 & 0.044 \\ 
  $L_{\brho,\log}$ & 0.065 & 0.05 & 0.055 & 0.056 & 0.052 & 0.045 & - & - & - \\ 
  $L_{\brho,\max}$ & 0.023 & 0.029 & 0.044 & 0.03 & 0.031 & 0.039 & 0.016 & 0.02 & 0.038 \\ 
  $L_{\tbrho,2}$ & 0.058 & 0.058 & 0.04 & 0.047 & 0.058 & 0.041 & 0.057 & 0.057 & 0.05 \\ 
  $L_{\tbrho,\log}$ & 0.121 & 0.071 & 0.064 & 0.072 & 0.061 & 0.049 & - & - & - \\ 
  $L_{\tbrho,\max}$ & 0.023 & 0.027 & 0.044 & 0.03 & 0.031 & 0.039 & 0.016 & 0.02 & 0.038 \\ 
  $L_{\K,2}$ & 0.06 & 0.058 & 0.041 & 0.043 & 0.06 & 0.045 & 0.056 & 0.058 & 0.052 \\ 
  $L_{\K,\log}$ & 0.062 & 0.063 & 0.053 & 0.057 & 0.051 & 0.05 & 0.125 & 0.074 & 0.072 \\ 
  $L_{\K,\max}$ & 0.031 & 0.041 & 0.051 & 0.037 & 0.037 & 0.042 & 0.029 & 0.032 & 0.038 \\ 
  $L_{\R,2}$ & 1 & 1 & 1 & 1 & 1 & 1 & 1 & 1 & 1 \\ 
  $L_{\R,\log}$ & 0.912 & 0.983 & 1 & 0.765 & 0.917 & 0.987 & - & - & - \\ 
  $L_{\R,\max}$ & 1 & 1 & 1 & 1 & 1 & 1 & 1 & 1 & 1 \\ 
\hline \\[-1.8ex] 
  \multicolumn{10}{c}{Mixed distribution}\\
  $L_{\brho,2}$ & 0.053 & 0.061 & 0.051 & 0.038 & 0.054 & 0.049 & 0.043 & 0.041 & 0.044 \\ 
  $L_{\brho,\log}$ & 0.057 & 0.052 & 0.055 & 0.056 & 0.053 & 0.05 & - & - & - \\ 
  $L_{\brho,\max}$ & 0.02 & 0.033 & 0.036 & 0.021 & 0.034 & 0.031 & 0.017 & 0.028 & 0.032 \\ 
  $L_{\tbrho,2}$ & 0.056 & 0.063 & 0.052 & 0.042 & 0.054 & 0.049 & 0.049 & 0.043 & 0.045 \\ 
  $L_{\tbrho,\log}$ & 0.109 & 0.068 & 0.065 & 0.077 & 0.063 & 0.056 & - & - & - \\ 
  $L_{\tbrho,\max}$ & 0.02 & 0.033 & 0.036 & 0.021 & 0.034 & 0.031 & 0.017 & 0.028 & 0.032 \\ 
  $L_{\K,2}$ & 0.06 & 0.064 & 0.054 & 0.047 & 0.061 & 0.048 & 0.049 & 0.042 & 0.044 \\ 
  $L_{\K,\log}$ & 0.065 & 0.062 & 0.059 & 0.052 & 0.065 & 0.049 & 0.115 & 0.066 & 0.061 \\ 
  $L_{\K,\max}$ & 0.036 & 0.047 & 0.043 & 0.03 & 0.043 & 0.036 & 0.033 & 0.044 & 0.04 \\ 
  $L_{\R,2}$ & 1 & 1 & 1 & 1 & 1 & 1 & 1 & 1 & 1 \\ 
  $L_{\R,\log}$ & 1 & 1 & 1 & 1 & 1 & 1 & - & - & - \\ 
  $L_{\R,\max}$ & 0.91 & 1 & 1 & 0.712 & 0.988 & 1 & 1 & 1 & 1 \\ 
  \hline \\[-1.8ex]
  \end{tabular} }
  \end{table}

To evaluate the power of our proposed test statistics, we generate data under various alternative hypotheses. We start with data generated from the above three null distributions and then generate the correlated data $\X_n$ as follows:
\begin{itemize}
    \item Global correlation: $\X_n=\Z_n\bSig$, where $\bSig=\left(\sigma_{ij}\right)_{p\times p}$ is the Toeplitz matrix with $\sigma_{ii}=1$, $\sigma_{i-1,i}=\sigma_{i,i+1}=\rho$, $\sigma_{ij}=0$ for $|i-j|>1$;
    \item Local correlation: $X_{i1}=Z_{i1}+\rho Z_{i2}$, $X_{i2}=\rho Z_{i1}+Z_{i2}$ and $X_{ij}=Z_{ij}$ for $1\leq i\leq n$, $3\leq j\leq p$.
\end{itemize}
By combining these data generation methods, we obtain six different alternative hypotheses. For each alternative hypothesis, we fix the sample size ($n=200$) and dimension ($p=100$), and vary the correlation parameters $\rho$ to assess the power of the tests. Since Pearson's correlation-based tests are sensitive to distributional assumptions, we focus on rank-based tests (Spearman and Kendall) in our power analysis. Simulation results are presented in Table \ref{table:emp_power_1} and Table \ref{table:emp_power_2}.

\begin{table}[ht!]  
    \centering
    \caption{Empirical powers of independence test statistics based on Spearman and Kendall's correlations under global structure.}
    \label{table:emp_power_1} 
    \resizebox{\textwidth}{!}{
  \begin{tabular}{@{\extracolsep{5pt}} ccccccccc} 
  \\[-1.8ex]\hline 
  \hline \\[-1.8ex]
  $\rho$ & 0.01 & 0.02 & 0.03 & 0.04 & 0.05 & 0.06 & 0.07 & 0.08 \\
  \hline \\[-1.8ex] 
  \multicolumn{9}{c}{Normal distribution}\\
  $L_{\brho,2}$ & 0.06 & 0.071 & 0.175 & 0.295 & 0.538 & 0.804 & 0.958 & 0.995 \\ 
  $L_{\brho,\log}$ & 0.065 & 0.063 & 0.137 & 0.243 & 0.417 & 0.674 & 0.864 & 0.969 \\ 
  $L_{\brho,\max}$ & 0.031 & 0.04 & 0.041 & 0.071 & 0.107 & 0.18 & 0.311 & 0.584 \\ 
  $L_{\tbrho,2}$ & 0.06 & 0.072 & 0.178 & 0.3 & 0.54 & 0.807 & 0.958 & 0.995 \\ 
  $L_{\tbrho,\log}$ & 0.071 & 0.074 & 0.145 & 0.264 & 0.443 & 0.683 & 0.884 & 0.971 \\ 
  $L_{\tbrho,\max}$ & 0.031 & 0.04 & 0.041 & 0.071 & 0.107 & 0.179 & 0.31 & 0.583 \\ 
  $L_{\K,2}$ & 0.061 & 0.075 & 0.178 & 0.31 & 0.543 & 0.804 & 0.956 & 0.996 \\ 
  $L_{\K,\log}$ & 0.063 & 0.08 & 0.173 & 0.298 & 0.503 & 0.782 & 0.948 & 0.993 \\ 
  $L_{\K,\max}$ & 0.037 & 0.037 & 0.051 & 0.073 & 0.122 & 0.203 & 0.356 & 0.616 \\ 
  \hline \\[-1.8ex] 
  \multicolumn{9}{c}{Cauchy distribution}\\
$L_{\brho,2}$ & 0.241 & 0.836 & 0.997 & 1 & 1 & 1 & 1 & 1 \\ 
$L_{\brho,\log}$ & 0.183 & 0.693 & 0.982 & 1 & 1 & 1 & 1 & 1 \\ 
$L_{\brho,\max}$ & 0.057 & 0.224 & 0.713 & 0.987 & 1 & 1 & 1 & 1 \\ 
$L_{\tbrho,2}$ & 0.247 & 0.841 & 0.997 & 1 & 1 & 1 & 1 & 1 \\ 
$L_{\tbrho,\log}$ & 0.196 & 0.707 & 0.983 & 1 & 1 & 1 & 1 & 1 \\ 
$L_{\tbrho,\max}$ & 0.057 & 0.223 & 0.713 & 0.987 & 1 & 1 & 1 & 1 \\ 
$L_{\K,2}$ & 0.248 & 0.852 & 0.999 & 1 & 1 & 1 & 1 & 1 \\ 
$L_{\K,\log}$ & 0.232 & 0.829 & 0.998 & 1 & 1 & 1 & 1 & 1 \\ 
$L_{\K,\max}$ & 0.071 & 0.279 & 0.791 & 0.999 & 1 & 1 & 1 & 1 \\ 
\hline \\[-1.8ex] 
\multicolumn{9}{c}{Mixed distribution}\\
$L_{\brho,2}$ & 0.079 & 0.256 & 0.6 & 0.882 & 0.992 & 1 & 1 & 1 \\ 
$L_{\brho,\log}$ & 0.077 & 0.182 & 0.461 & 0.763 & 0.964 & 0.999 & 1 & 1 \\ 
$L_{\brho,\max}$ & 0.038 & 0.083 & 0.293 & 0.697 & 0.954 & 0.996 & 1 & 1 \\ 
$L_{\tbrho,2}$ & 0.081 & 0.261 & 0.603 & 0.883 & 0.992 & 1 & 1 & 1 \\ 
$L_{\tbrho,\log}$ & 0.083 & 0.203 & 0.477 & 0.778 & 0.968 & 0.999 & 1 & 1 \\ 
$L_{\tbrho,\max}$ & 0.038 & 0.083 & 0.292 & 0.697 & 0.954 & 0.996 & 1 & 1 \\ 
$L_{\K,2}$ & 0.081 & 0.263 & 0.622 & 0.889 & 0.994 & 1 & 1 & 1 \\ 
$L_{\K,\log}$ & 0.085 & 0.251 & 0.582 & 0.874 & 0.993 & 1 & 1 & 1 \\ 
$L_{\K,\max}$ & 0.053 & 0.106 & 0.355 & 0.765 & 0.971 & 0.999 & 1 & 1 \\ 
\hline \\[-1.8ex] 
  \end{tabular} }
  \end{table} 

\begin{table}[ht!]
    \centering
    \caption{Empirical powers of independence test statistics based on Spearman and Kendall's correlations under local structure.}
    \label{table:emp_power_2} 
    \resizebox{\textwidth}{!}{
  \begin{tabular}{@{\extracolsep{5pt}} ccccccccc} 
  \\[-1.8ex]\hline 
  \hline \\[-1.8ex]
  $\rho$ & 0.1 & 0.2 & 0.3 & 0.4 & 0.5 & 0.6 & 0.7 & 0.8 \\
  \hline \\[-1.8ex] 
  \multicolumn{9}{c}{Normal distribution}\\
  $L_{\brho,2}$ & 0.059 & 0.085 & 0.13 & 0.212 & 0.35 & 0.432 & 0.512 & 0.594 \\ 
  $L_{\brho,\log}$ & 0.054 & 0.081 & 0.114 & 0.25 & 0.504 & 0.715 & 0.937 & 0.998 \\ 
  $L_{\brho,\max}$ & 0.059 & 0.818 & 1 & 1 & 1 & 1 & 1 & 1 \\ 
  $L_{\tbrho,2}$ & 0.06 & 0.089 & 0.131 & 0.219 & 0.357 & 0.436 & 0.516 & 0.595 \\ 
  $L_{\tbrho,\log}$ & 0.066 & 0.09 & 0.131 & 0.269 & 0.528 & 0.738 & 0.945 & 0.999 \\ 
  $L_{\tbrho,\max}$ & 0.059 & 0.818 & 1 & 1 & 1 & 1 & 1 & 1 \\ 
  $L_{\K,2}$ & 0.062 & 0.091 & 0.144 & 0.267 & 0.492 & 0.676 & 0.829 & 0.939 \\ 
  $L_{\K,\log}$ & 0.057 & 0.092 & 0.154 & 0.314 & 0.628 & 0.887 & 0.994 & 1 \\ 
  $L_{\K,\max}$ & 0.068 & 0.825 & 1 & 1 & 1 & 1 & 1 & 1 \\ 
  \hline \\[-1.8ex] 
  \multicolumn{9}{c}{Cauchy distribution}\\
$L_{\brho,2}$ & 0.092 & 0.164 & 0.261 & 0.384 & 0.456 & 0.513 & 0.577 & 0.626 \\ 
$L_{\brho,\log}$ & 0.083 & 0.167 & 0.326 & 0.537 & 0.799 & 0.941 & 0.99 & 1 \\ 
$L_{\brho,\max}$ & 0.866 & 1 & 1 & 1 & 1 & 1 & 1 & 1 \\ 
$L_{\tbrho,2}$ & 0.097 & 0.167 & 0.268 & 0.388 & 0.456 & 0.513 & 0.577 & 0.626 \\ 
$L_{\tbrho,\log}$ & 0.091 & 0.176 & 0.356 & 0.557 & 0.82 & 0.948 & 0.992 & 1 \\ 
$L_{\tbrho,\max}$ & 0.866 & 1 & 1 & 1 & 1 & 1 & 1 & 1 \\ 
$L_{\K,2}$ & 0.103 & 0.206 & 0.372 & 0.575 & 0.774 & 0.887 & 0.944 & 0.976 \\ 
$L_{\K,\log}$ & 0.109 & 0.232 & 0.483 & 0.778 & 0.958 & 0.999 & 1 & 1 \\ 
$L_{\K,\max}$ & 0.894 & 1 & 1 & 1 & 1 & 1 & 1 & 1 \\ 
\hline \\[-1.8ex] 
\multicolumn{9}{c}{Mixed distribution}\\
$L_{\brho,2}$ & 0.088 & 0.173 & 0.288 & 0.363 & 0.477 & 0.502 & 0.583 & 0.615 \\ 
$L_{\brho,\log}$ & 0.084 & 0.177 & 0.335 & 0.561 & 0.775 & 0.937 & 0.99 & 1 \\ 
$L_{\brho,\max}$ & 0.862 & 1 & 1 & 1 & 1 & 1 & 1 & 1 \\ 
$L_{\tbrho,2}$ & 0.091 & 0.176 & 0.293 & 0.365 & 0.481 & 0.504 & 0.583 & 0.616 \\ 
$L_{\tbrho,\log}$ & 0.095 & 0.185 & 0.353 & 0.577 & 0.795 & 0.952 & 0.992 & 1 \\ 
$L_{\tbrho,\max}$ & 0.862 & 1 & 1 & 1 & 1 & 1 & 1 & 1 \\ 
$L_{\K,2}$ & 0.098 & 0.216 & 0.393 & 0.586 & 0.742 & 0.879 & 0.946 & 0.983 \\ 
$L_{\K,\log}$ & 0.102 & 0.245 & 0.491 & 0.773 & 0.955 & 0.999 & 1 & 1 \\ 
$L_{\K,\max}$ & 0.896 & 1 & 1 & 1 & 1 & 1 & 1 & 1 \\ 
\hline \\[-1.8ex] 
\end{tabular} }
\end{table}

From Tables \ref{table:emp_power_1} and \ref{table:emp_power_2}, we observe that the power of all test statistics increases as the correlation strength $\rho$ increases. This demonstrates the effectiveness of rank-based tests, especially under heavy-tailed distributions. For global correlation structures, tests based on $\ell_2$ loss and Stein's loss tend to have higher power. In contrast, for local correlation structures, maximum-type tests exhibit better power. We leave the theoretical analysis of these powers as a future work. Overall, our proposed test statistics $L_{\brho,\log}$, $L_{\tbrho,2}$, and $L_{\tbrho,\log}$ demonstrate comparable performance across various scenarios.

\section{Discussion}
For the Stieltjes transform of the considered matrix, we can get study its limit which yields the limit of LSS and the CLT which yields the CLT of LSS. We summarize the results as following for $\Im(z)>c$. 
\begin{itemize}
    \item For the Gram matrix $\g_n \in \mR^{n \times n}$, we have
    \begin{align*}
        & \mbox{LSD}:~F^{\g_n} \tod  F_{1/y},~a.s.;\\
        & \mbox{Stieltjes transform:}~s_n(z)=\frac{1}{n}\tr(\g_n-z \bI_n)^{-1} \toas s(z)=m(1/y,z);\\
        & \mbox{CLT:}~n\left(s_n(z)-\E s_n(z) \right) \tod \mbox{Gaussian Processes} \left(0, \sigma(z_1,z_2)\right),
    \end{align*}
    where 
    \begin{align*}
      n \E s_n(z)=n \cdot m(n/p,z)+\mu(z)+o(1).
    \end{align*}
    \item For re-scaled Spearman's rank correlation matrix $\brho_n/\y_n \in \mR^{p \times p}$,
        \begin{align*}
        & \mbox{LSD}:~F^{\brho_n/\y_n} \tod  \underline{F}_{1/y},~a.s.;\\
        & \mbox{Stieltjes transform:}~\underline{s}_n(z)=\frac{1}{p}\tr(\brho_n/\y_n-z \bI_n)^{-1}=\frac{n}{p}\left(s_n(z)+\frac{1}{z}\right)-\frac{1}{z}\\
        &~~~~~~~~~~~~~~~~~~~~~~~~~~~~~~~~~\toas \underline{s}(z)=\frac{1}{y}\left(s(z)+\frac{1}{z}\right)-\frac{1}{z};\\
       & \mbox{CLT:}~p \left(\underline{s}_n(z)-\E \underline{s}_n(z) \right) \tod \mbox{Gaussian Processes} \left(0, \sigma(z_1,z_2)\right),
    \end{align*}
    where 
    \begin{align*}
      p \E \underline{s}_n(z)=&n \cdot m(n/p,z)+\mu(z)+\frac{n-p}{z}+o(1)\\
      =&n\left(m(n/p,z)+(1-\frac{p}{n})\frac{1}{z}\right)+\mu(z)+o(1).
    \end{align*}
       \item For Spearman's rank correlation matrix $\brho_n\in \mR^{p \times p}$
    \begin{align*}
        & \mbox{LSD}:~F^{\brho_n} \tod  F_{y},~a.s.;\\
        & \mbox{Stieltjes transform:}~m_n(z)=\frac{1}{p}\tr(\brho_n-z \bI_n)^{-1}=\frac{1}{y_n} \underline{s}_n(z/y_n)\\
        &~~~~~~~~~~~~~~~~~~~~~~~~~~~~~~~~~~\toas \frac{1}{y} \underline{s}(z/y)=m(z)=m(y,z);\\
         &     \mbox{CLT:}~p\left(m_n(z)-\E m_n(z) \right) \tod \mbox{Gaussian Processes} \left(0, \frac{\sigma(z_1/y,z_2/y)}{y^2}\right),
    \end{align*}
    where 
    \begin{align*}
      p \E m_n(z)=p \cdot m(p/n,z)+\frac{\mu(z/y)}{y}+o(1).
    \end{align*}
%

           \item For improved Spearman's rank correlation matrix $\tbrho_n \in \mR^{p \times p}$
    \begin{align*}
      &  \mbox{LSD}:~F^{\tbrho_n} \tod  F_{y},~a.s.;\\
      &  \mbox{Stieltjes transform:}~\widetilde{m}_n(z)=\frac{1}{p}\tr(\tbrho_n-z \bI_n)^{-1} \toas m(z)=m(y,z);\\
    & \mbox{CLT:}~p\left(\widetilde{m}_n(z)-\E \widetilde{m}_n(z) \right) \tod \mbox{Gaussian Processes} \left(0, \frac{\sigma(z_1/y,z_2/y)}{y^2}\right),
    \end{align*}
    where 
    \begin{align*}
      p \E \widetilde{m}_n(z)=p \cdot m(p/n,z)+\frac{\mu(z/y)}{y}+\frac{\widetilde{\mu}(z/y)}{y} +o(1).
    \end{align*}
\end{itemize}
With these CLTs, we can construct hypothesis tests based on Spearman's and improved Spearman's correlation matrices. Our simulation studies demonstrate the practical applicability of these new test statistics.

In this work, we study the improved Pearson's correlation which is a standard U-statistic of order 3. Studying general U-statistic typed correlation matrices could be a topic of future work. Moreover, we compare the test statistics through simulations. Investigating the asymptotic distribution of test statistics under local alternatives could be another interesting future work.

\section{Proof}
\subsection{Proof sketch}
The rigorous proof of Theorem \ref{thm:LSS_clt} and Theorem \ref{thm:LSS_clt_3} will be presented in this section. As can be seen, the Gram-type of Spearman's rank correlation matrix $\g_n$ is formulated as the sum of independent outer product matrices, which reveals similar structure with sample covariance matrix. Therefore, our methodology is originated from the proof in \cite{bai2004clt} that establish the CLT of LSS of large dimensional sample covariance matrices. 

We denote $s_n^{(0)}(z)$ and $\underline{s}_n^{(0)}(z)$ as the Stieltjes transforms of $F_{n/p}$ and $\underline{F}_{n/p}$ respectively. By Cauchy's integral formula, we have
\begin{align}
    \int f(x)dG_n(x)=-\frac{1}{2\pi i}\int f(z)\cdot n\left(s_n(z)-s_n^{(0)}(z)\right)dz,\label{form:cauchy_integral_formula}
\end{align}
where the contour of this integral is closed and enclose the extreme eigenvalues of $\g_n$. It is noted that if the limit superior and limit inferior of extreme eigenvalues of $\g_n$ are contained in the support of $F_{y_0}$ with probability $1$, then for any function $f$ analytic on \eqref{form:support} and closed contour enclosing \eqref{form:support}, the formula \eqref{form:cauchy_integral_formula} holds for all sufficiently large $n$ with probability 1. However the concerntration of extreme eigenvalues are not trivial at all, and a more stronger control is presented in the following lemma.
\begin{lem}\label{lem:extreme_eigen_control}
    Under the same assumptions in Theorem \ref{thm:LSS_clt}, for any $\eta_l<(1-\sqrt{y})^2$, $\eta_r>(1+\sqrt{y})^2$ and any $m>0$,
    \begin{align}
        P\left(\lambda_{1}(\brho_{n})>\eta_{r}\right)=o(n^{-m}),\quad P\left(\lambda_{\min\{n,p\}}(\brho_{n})\leq\eta_{l}\right)=o(n^{-m}).
    \end{align}
\end{lem}
\begin{remark}
    The boundness of the largest eigenvalue is a direct corollary of Proposition 2.3 in \cite{bao2019tracy_spearman}. Due to the strong local law, the rigidity on the left edge can be derived with the same steps as right edge. Therefore, the boundness of the smallest eigenvalue can also be concluded.
\end{remark}
As has been discussed in the above, the focus of the problem is shifted to establishing the asymptotic distribution of 
\begin{align*}
    M_n(z)\defby n\left(s_n(z)-s_n^{(0)}(z)\right)=p\left(\underline{s}_n(z)-\underline{s}_n^{(0)}(z)\right).
\end{align*}
Since the CLT of LSS is obtained through a process of integration, we define a contour $\CC$ enclosing interval \eqref{form:support} as follows. Let $\eta_l$ and $\eta_r$ be any two numbers such that $\eqref{form:support}\subset(\eta_l,\eta_r)$, and choose $v_0>0$. The contour is described as a rectangle,
\begin{align*}
    \CC=\{x\pm iv_0:x\in[\eta_l,\eta_r]\}\cup\{x+iv:x\in\{\eta_l,\eta_r\},v\in[-v_0,v_0]\}.
\end{align*}
For further analysis, we consider $\widehat{M}_n(z)$ instead, a truncated version of $M_n(z)$, which is defined as 
\begin{align*}
    \widehat{M}_n(z)=\begin{cases}
        M_n(z),&z\in\CC_n,\\
        M_n(x+in^{-1}\varepsilon_n),&x\in\{\eta_l,\eta_r\}\text{ and }v\in[0,n^{-1}\varepsilon_n],\\
        M_n(x-in^{-1}\varepsilon_n),&x\in\{\eta_l,\eta_r\}\text{ and }v\in[-n^{-1}\varepsilon_n,0],
    \end{cases}
\end{align*}
where $\CC_n=\{x\pm iv_0:x\in[\eta_l,\eta_r]\}\cup\{x\pm iv:x\in\{\eta_l,\eta_r\},v\in[n^{-1}\varepsilon_n,v_0]\}$ and $\{\varepsilon_n\}$ is a sequence decreasing to zero satisfying $\varepsilon_n\geq n^{-\alpha}$ for some $\alpha\in(0,1)$. It follows that $\widehat{M}_n(z)$ pauses at $x+in^{-1}\varepsilon_n$ when $z$ tends to the real line, which makes the imaginary gap a natural bound to control the spectral norm or Euclidean distance of Stieltjes transforms. Besides, this truncation step have no influence on the limiting behavior of \eqref{form:cauchy_integral_formula} since for all sufficiently large $n$,
\begin{align} \label{form:Mn_hat}
    \left|\int_{\CC} f(z)\left(M_n(z)-\widehat{M}_n(z)\right)dz\right|\leq& C\varepsilon_n\left(\left|(1+\sqrt{y_0})^2\vee\lambda_1(\g_n)-\eta_r\right|\right.\\
    &\left.+\left|I(y_n>1)(1-\sqrt{y_0})^2\wedge\lambda_p(\g_n)-\eta_l\right|\right),\notag
\end{align}
which converges to zero. So now we have prepared all ingredients and the proof of Theorem \ref{thm:LSS_clt} can be completed by the following lemma establishing the convergence of $\widehat{M}_n(z)$ on $\CC$.
\begin{lem}\label{lem:clt_Mn(z)}
    Under the same assumptions in Theorem \ref{thm:LSS_clt}, $\{\widehat{M}_n(\cdot)\}$, as a stochastic process on $\CC$, converges weakly to a Gaussian process $M(\cdot)$ with mean function
    \begin{align*}
        \E M(z)=\mu(z),
    \end{align*}
    and covariance function
    \begin{align*}
        \cov(M(z_1),M(z_2))=\sigma(z_1,z_2),
    \end{align*}
    where $\mu(z)$ and $\sigma(z_1,z_2)$ are defined in Theorem \ref{thm:LSS_clt}.
\end{lem}
Theorem \ref{thm:LSS_clt_2} is a corollary of Theorem \ref{thm:LSS_clt} with the application of Cauchy's integral formula. More specifically, 
if $f(y\cdot)$ is analytic on an open interval containing \eqref{form:support}, $f(y_n\cdot)$ converges to $f(y\cdot)$ uniformly and by the method of Stieltjes transform,
\begin{align*}
    T(f)=&-\frac{1}{2\pi i}\int_{\CC} f(y_nz)M_n(z)dz\\
    =&-\frac{1}{2\pi i}\int_{\CC} f(y_nz)\widehat{M}_n(z)dz+o_P(1)\\
    =&-\frac{1}{2\pi i}\int_{\CC} f(yz)\widehat{M}_n(z)dz+o_P(1),
\end{align*}
where the contour $\CC$ encloses a neighborhoog of \eqref{form:support}. The second equality holds by the same procedure of \eqref{form:Mn_hat}, and the last equality holds by 
\begin{align*}
    \E\left|\int_{\CC} \left(f(y_nz)-f(yz)\right)\widehat{M}_n(z)dz\right|\leq&|\CC|\cdot\sup_{z\in\CC}\left|f(y_nz)-f(yz)\right|\cdot\sup_{z\in\CC}\E\left|\widehat{M}_n(z)\right|\\
    =&o(1).
\end{align*}
Therefore, by the convergence of $\widehat{M}_n(z)$ stated in Lemma \ref{lem:clt_Mn(z)}, we obtain Theorem \ref{thm:LSS_clt_2}.

The proof of Theorem \ref{thm:LSS_clt_3} basically follows the same approach as Lemma \ref{lem:clt_Mn(z)}. With the help of Theorem \ref{thm:LSS_clt_2}, we only need to figure out the difference of $F^{\tbrho_n}$ and $F^{\brho_n}$. By Cauchy's integral formula,
\begin{align*}
    p\left(\int f(x)dF^{\tbrho_n}(x)-\int f(x)dF^{\brho_n}(x)\right)=-\frac{1}{2\pi i}\int f(z)\cdot p\left(m_{F^{\tbrho_n}}(z)-m_{F^{\brho_n}}(z)\right)dz,
\end{align*}
and we are supposed to find the limit of
\begin{align*}
    L_n(z)\defby p\left(m_{F^{\tbrho_n}}(z)-m_{F^{\brho_n}}(z)\right). 
\end{align*}
As has been discussed before, we consider $\widehat{L}_n(z)$, a truncated version of $L_n(z)$, defined on $\CC$ by the same way of $\widehat{M}_n(z)$. By the rigidity of the edge of Kendall's correlation matrix, we are able to control the extreme eigenvalues of $\tbrho_n$ as in Lemma \ref{lem:extreme_eigen_control}, which implies $\int f(z)L_n(z)dz-\int f(z)\widehat{L}_n(z)dz\to 0$ almost surely. The following Lemma states the convergence of $\widehat{L}_n(z)$, which conclude Theorem \ref{thm:LSS_clt_3}.

\begin{lem}\label{lem:clt_Ln(z)}
    Under the same assumptions in Theorem \ref{thm:LSS_clt_3}, $\{\widehat{L}_n(\cdot)\}$, as a stochastic process on $\CC$, converges weakly to a non-random function 
    \begin{align*}
        L(z)=\widetilde{\mu}(z/y)/y,
    \end{align*}
    where $\widetilde{\mu}(z)$ is defined in Theorem \ref{thm:LSS_clt_3}.
\end{lem}

\subsection{Proof of Lemma \ref{lem:clt_Mn(z)}}
We decompose $\widehat{M}_n(z)$ into two parts as 
\begin{align*}
    \widehat{M}_n(z)=&n\left(s_n(z)-\E s_n(z)\right)+n\left(\E s_n(z)-s_n^{(0)}(z)\right)\\
    \defby &M_n^{(1)}(z)+M_n^{(2)}(z),
\end{align*}
where $M_n^{(1)}(z)$ is the random part and $M_n^{(2)}(z)$ is the non-random part. For simplicity, denote 
\begin{align*}
    \D(z)=&\g_n-z\bI_n,\quad\D_j(z)=\g_n-z\bI_n-\frac{1}{p}\s_j\s_j\trans,\\
    \beta_j(z)=&\frac{1}{1+\frac{1}{p}\s_j\trans\D_j^{-1}(z)\s_j},\quad\overline{\beta}_j(z)=\frac{1}{1+\frac{1}{p}\tr\bSig\D_j^{-1}(z)},\quad b_n(z)=\frac{1}{1+\frac{1}{p}\E\tr\bSig\D_1^{-1}(z)},\\
    \varepsilon_j(z)=&\frac{1}{p}\s_j\trans\D_j^{-1}(z)\s_j-\frac{1}{p}\tr\bSig\D_j^{-1}(z),\quad\delta_j(z)=\frac{1}{p}\s_j\trans\D_j^{-2}(z)\s_j-\frac{1}{p}\E\tr\bSig\D_j^{-2}(z).
\end{align*}
Note that $\beta_j(z)$, $\overline{\beta}_j(z)$ and $b_n(z)$ are all bounded by $\frac{|z|}{\Im(z)}$, where $\Im(\cdot)$ is the imaginary part. And by some matrix identity, we have
\begin{align}
    \D_j^{-1}(z)=&\D^{-1}(z)+\frac{1}{p}\beta_j(z)\trans\D_j^{-1}(z)\s_j\s_j\trans\D_j^{-1}(z),\label{form:leave_one_out}\\
    \underline{s}_n(z)=&-\frac{1}{z}\cdot\frac{1}{p}\sum_{i=1}^p\beta_i(z).\label{form:relation_to_zm(z)}
\end{align}

\textbf{Step 1. Finite-dimensional weak convergence of $M_n^{(1)}(z)$.}

Let $\E_k(\cdot)$ be the conditional expectation with respect to the $\sigma$-field generated by $\s_1,\cdots,\s_k$. By martingale difference decomposition,
\begin{align*}
    M_n^{(1)}(z)=&\sum_{j=1}^p\left(\E_j\tr\D^{-1}(z)-\E_{j-1}\tr\D^{-1}(z)\right)\\
    =&-\frac{1}{p}\sum_{j=1}^p\left(\E_j-\E_{j-1}\right)\beta_j(z)\s_j\trans\D_j^{-2}(z)\s_j\\
    =&-\frac{1}{p}\sum_{j=1}^p\left(\E_j-\E_{j-1}\right)\left(\overline{\beta}_j(z)-\overline{\beta}_j^2(z)\varepsilon_j(z)+\overline{\beta}_j^2(z)\beta_j(z)\varepsilon_j^2(z)\right)\s_j\trans\D_j^{-2}(z)\s_j\\
    =&-\sum_{j=1}^p\E_j\left(\overline{\beta}_j(z)\delta_j(z)+\overline{\beta}_j^2(z)\varepsilon_j(z)\frac{1}{p}\tr\bSig\D_j^{-2}(z)\right)+o_P(1),
\end{align*}
where the second-to-last equality holds by the identity $\beta_j(z)=\overline{\beta}_j(z)-\overline{\beta}_j^2(z)\varepsilon_j(z)+\overline{\beta}_j^2(z)\beta_j(z)\varepsilon_j^2(z)$, and the last equality holds by Lemma \ref{lem:higher_moments_control}.
The dominant term denoted by 
\begin{align*}
    Y_j(z)=\E_j\left(\overline{\beta}_j(z)\delta_j(z)+\overline{\beta}_j^2(z)\varepsilon_j(z)\frac{1}{p}\tr\bSig\D_j^{-2}(z)\right),\quad j=1,\cdots,p,
\end{align*}
is still a martingale difference sequence. For some fixed $r>0$, since
\begin{align*}
    \sum_{i=1}^r\alpha_iM_n^{(1)}(z_i)=\sum_{j=1}^p\sum_{i=1}^r\alpha_iY_j(z_i)+o_P(1),
\end{align*}
by martingale CLT in Lemma \ref{lem:martingale_clt}, it suffices to verify
\begin{align}
    \sum_{j=1}^p\E\left|\sum_{i=1}^m\alpha_iY_j(z_i)\right|^4\to 0,\label{form:martingale_clt_condition1}
\end{align}
and find the limit of convergence in probability of 
\begin{align}
    \sum_{j=1}^n\E_{j-1}Y_j(z_1)Y_j(z_2).\label{form:martingale_clt_condition2}
\end{align}
By Lemma \ref{lem:higher_moments_control},
\begin{align*}
    \E\left|Y_j(z)\right|^4\leq C\left(\frac{|z|^4}{\Im^4(z)}\E|\delta_j(z)|^4+\frac{|z|^8}{\Im^{16}(z)}\E|\varepsilon_j(z)|^4\right)=o(n^{-1}),
\end{align*}
which implies \eqref{form:martingale_clt_condition1}.
As for \eqref{form:martingale_clt_condition2}, observe that
\begin{align*}
    \overline{\beta}_j(z)\delta_j(z)+\overline{\beta}_j^2(z)\varepsilon_j(z)\frac{1}{p}\tr\bSig\D_j^{-2}(z)=\frac{d}{dz}\overline{\beta}_j(z)\varepsilon_j(z),
\end{align*}
so we have 
\begin{align*}
    \frac{\partial^2}{\partial z_1\partial z_2}\E_{j-1}\left[\E_j\left(\overline{\beta}_j(z_1)\varepsilon_j(z_1)\right)\E_j\left(\overline{\beta}_j(z_2)\varepsilon_j(z_2)\right)\right]=\E_{j-1}Y_j(z_1)Y_j(z_2).
\end{align*}
With similar arguments on Page 571 of \cite{bai2004clt}, it suffices to determine the limit of 
\begin{align}
    \sum_{j=1}^p\E_{j-1}\left[\E_j\left(\overline{\beta}_j(z_1)\varepsilon_j(z_1)\right)\E_j\left(\overline{\beta}_j(z_2)\varepsilon_j(z_2)\right)\right].\label{form:partial_tmp}
\end{align}
Since by \eqref{form:leave_one_out} and \eqref{form:relation_to_zm(z)}, we have
\begin{align*}
    \E\left|\overline{\beta}_j(z)-b_n(z)\right|^2\leq\frac{C}{p^2}\E\left|\tr\bSig\D_j^{-1}(z)-\E\tr\bSig\D_1^{-1}(z)\right|^2
    =O(p^{-1}),
\end{align*}
and 
\begin{align*}
    \left|b_n(z)+z\underline{s}(z)\right|\leq\left|b_n(z)-\E\beta_1(z)\right|+\left|\E\beta_1(z)+z\underline{s}(z)\right|\\
    =o(1).
\end{align*}
Therefore, we only need find the limit of 
\begin{align}
    z_1z_2\underline{s}(z_1)\underline{s}(z_2)\sum_{j=1}^p\E_{j-1}\left[\E_j\varepsilon_j(z_1)\E_j\varepsilon_j(z_2)\right].\label{form:partial}
\end{align}
By Lemma \ref{lem:second_moment} and \eqref{form:leave_one_out},
\begin{align*}
    z_1z_2\underline{s}(z_1)\underline{s}(z_2)\sum_{j=1}^p\E_{j-1}\E_j\varepsilon_j(z_1)\E_j\varepsilon_j(z_2)=2I_1-\frac{6}{5}I_2-\frac{4}{5}I_3+O_P(p^{-1}),
\end{align*}
where 
\begin{align*}
    I_1=&z_1z_2\underline{s}(z_1)\underline{s}(z_2)\frac{1}{p^2}\sum_{j=1}^p\tr\left(\E_j\D^{-1}(z_1)\E_j\D^{-1}(z_2)\right),\\
    I_2=&z_1z_2\underline{s}(z_1)\underline{s}(z_2)\frac{1}{p^2}\sum_{j=1}^p\tr\left(\E_j\D^{-1}(z_1)\circ\E_j\D^{-1}(z_2)\right),\\
    I_3=&z_1z_2\underline{s}(z_1)\underline{s}(z_2)\frac{1}{np^2}\sum_{j=1}^p\tr(\E_j\D^{-1}(z_1))\tr(\E_j\D^{-1}(z_2)).
\end{align*}
For $I_1$, since that 
\begin{align*}
    &\left|\tr\left(\E_j\D^{-1}(z_1)\E_j\D^{-1}(z_2)\right)-\tr\left(\bSig\E_j\D^{-1}(z_1)\E_j\D^{-1}(z_2)\right)\right|\\
    =&\left|-\frac{1}{n-1}\tr\left(\E_j\D^{-1}(z_1)\E_j\D^{-1}(z_2)\right)+\frac{1}{n-1}\one_n\trans\E_j\D^{-1}(z_1)\E_j\D^{-1}(z_2)\one_n\right|\\
    =&O(1),
\end{align*}
and similarly
\begin{align*}
    \left|\tr\left(\bSig\E_j\D^{-1}(z_1)\E_j\D^{-1}(z_2)\right)-\tr\left(\bSig\E_j\D^{-1}(z_1)\bSig\E_j\D^{-1}(z_2)\right)\right|=O(1),
\end{align*}
the effect of the multiplying $\bSig$ is negligible and 
\begin{align*}
    I_1=z_1z_2\underline{s}(z_1)\underline{s}(z_2)\frac{1}{p^2}\sum_{j=1}^p\tr\left(\bSig\E_j\D^{-1}(z_1)\bSig\E_j\D^{-1}(z_2)\right)+O(n^{-1}).
\end{align*}
With similar arguments on Pages 572-578, by Lemma \ref{lem:higher_moments_control},
\begin{align}
    I_1=\log\frac{\underline{s}(z_1)-\underline{s}(z_2)}{\underline{s}(z_1)\underline{s}(z_2)(z_1-z_2)}+o_P(1).\label{form:I_1_simplified}
\end{align}
For $I_2$, following similar steps on Pages 1247-1249 of \cite{pan2008central} with applications of Lemma \ref{lem:higher_moments_control}, we have
\begin{align}
    I_2=&z_1z_2\underline{s}(z_1)\underline{s}(z_2)\frac{1}{p^2}\sum_{j=1}^p\tr\left(\E\D^{-1}(z_1)\circ\E\D^{-1}(z_2)\right)+o_P(1)\label{form:start}\\
    =&\underline{s}(z_1)\underline{s}(z_2)\frac{1}{p^2}\sum_{j=1}^p\tr\left((\underline{s}(z_1)\bSig+\bI_n)^{-1}\circ(\underline{s}(z_2)\bSig+\bI_n)^{-1}\right)+o_P(1).\nonumber
\end{align} 
Denote $\e_k\in\mR^n$ as the unit vector with $k$-th element being one and the rest being zero. Since that
\begin{align*}
    &\left|\e_k\trans\left(\underline{s}(z)\bSig+\bI_n\right)^{-1}\e_k-(\underline{s}(z)+1)^{-1}\right|\\
    =&\left|\underline{s}(z)(\underline{s}(z)+1)^{-1}\e_k\trans\left(\left(\underline{s}(z)\bSig+\bI_n\right)^{-1}(\bI_n-\bSig)\right)\e_k\right|\\
    \leq&\frac{C}{n-1}\left|\e_k\trans\left(\underline{s}(z)\bSig+\bI_n\right)^{-1}\e_k\right|+\frac{C}{n-1}\left|\e_k\trans\left(\underline{s}(z)\bSig+\bI_n\right)^{-1}\one_n\one_n\trans\e_k\right|\\
    =&O(n^{-\frac{1}{2}}),
\end{align*}
$\bSig$ can be approximated by $\bI_n$ and 
\begin{align*}
    &\tr\left((\underline{s}(z_1)\bSig+\bI_n)^{-1}\circ(\underline{s}(z_2)\bSig+\bI_n)^{-1}\right)\\
    =&\sum_{k=1}^n\e_k\trans(\underline{s}(z_1)\bSig+\bI_n)^{-1}\e_k\e_k\trans(\underline{s}(z_2)\bSig+\bI_n)^{-1}\e_k\\
    =&n\left(\underline{s}(z_1)+1\right)^{-1}\left(\underline{s}(z_2)+1\right)^{-1}+O(1).
\end{align*}
Thus, we obtain 
\begin{align}
    I_2=\frac{y_0\underline{s}(z_1)\underline{s}(z_2)}{\left(\underline{s}(z_1)+1\right)\left(\underline{s}(z_2)+1\right)}+o_P(1).\label{form:I_2_simplified}
\end{align}
$I_3$ can be simplified by the following approximation steps,
\begin{align*}
    \tr\D^{-1}(z)=\tr\left(-z\underline{s}(z)\bSig-z\bI\right)^{-1}+o_P(1)=n\left(z\underline{s}+z\right)^{-1}+o_P(1),
\end{align*}
which implies 
\begin{align}
    I_3=\frac{y_0\underline{s}(z_1)\underline{s}(z_2)}{\left(\underline{s}(z_1)+1\right)\left(\underline{s}(z_2)+1\right)}+o_P(1).\label{form:I_3_simplified}
\end{align}
Collecting \eqref{form:I_1_simplified}, \eqref{form:I_2_simplified} and \eqref{form:I_3_simplified}, we have 
\begin{align*}
    \eqref{form:partial}\to2\log\frac{\underline{s}(z_1)-\underline{s}(z_2)}{\underline{s}(z_1)\underline{s}(z_2)(z_1-z_2)}-\frac{2y_0\underline{s}(z_1)\underline{s}(z_2)}{\left(\underline{s}(z_1)+1\right)\left(\underline{s}(z_2)+1\right)},\quad\text{in probability,}
\end{align*}
which conclude by taking derivatives
\begin{align}
    \eqref{form:martingale_clt_condition2}\to\sigma(z_1,z_2),\quad\text{in probability.}
\end{align}

\textbf{Step 2. Tightness of $M_n^{(1)}(z)$.}

Combined with finite-dimensional weak convergence of $M_n^{(1)}(z)$ and tightness on $z\in\CC_n$, we are able to prove the weak convergence of stochastic process $M_n^{(1)}(\cdot)$. To prove its tightness, by Theorem 12.3 of \cite{billingsley1968convergence}, it suffices to verify
\begin{align*}
    \sup_{n;z_1,z_2\in\CC_n}\frac{\E|M_n^{(1)}(z_1)-M_n^{(1)}(z_2)|^2}{|z_1-z_2|^2}<\infty.
\end{align*}
By Lemma \ref{lem:extreme_eigen_control}, 
\begin{align}
    \E\|\D^{-1}(z)\|^k\leq C_1+v^{-k}P\left(\|\G\|\geq\eta_r\text{ or }\lambda_{\min}(\G)\leq\eta_l\right)\leq C,\label{form:uniform_control}
\end{align}
for sufficiently large $l$. We emphasize that the moments bound here are uniform in $n$ and $z\in\CC_n$, that is, the constant $C$ is independent of $n$ and $z\in\CC_n$. By the same way in \eqref{form:uniform_control}, one can prove that the moments of $\|\D_j^{-1}(z)\|$ is also bounded uniformly in $n$ and $z\in\CC_n$. Therefore, we extend Lemma \ref{lem:higher_moments_control} slightly as
\begin{align}
    \left|\E a(v)\prod_{l=1}^q\left(\s_1\trans\B_l(v)\s_1-\frac{1}{p}\tr\bSig\B_l(v)\right)\right|\leq Cn^{-\frac{q}{2}+\delta},\label{form:higher_moment_uniform_control}
\end{align}
where $\B_l(v)$ is independent of $\s_1$ and $a(v)$ is some product of factors of the form $\beta_1(z)$ or $\s_1\trans\B_l(v)\s_1$. Following similar procedures on Pages 581-583 of \cite{bai2004clt} and applying \eqref{form:higher_moment_uniform_control}, we have
\begin{align*}
    \E\left|\frac{M_n^{(1)}(z_1)-M_n^{(1)}(z_2)}{z_1-z_2}\right|^2=\E\left|\tr\D^{-1}(z_1)\D^{-1}(z_2)-\E\tr\D^{-1}(z_1)\D^{-1}(z_2)\right|^2\leq C,
\end{align*}
uniformly in $z_1,z_2\in\CC_n$.

\textbf{Step 3. Uniform convergence of $M_n^{(2)}(z)$.}

Before proceeding, we collect some necessary results as follows, whose proofs are omitted since one can verify them in the same approaches on Pages 584-586 of \cite{bai2004clt}:
\begin{gather}
    \sup_{z\in\CC_n}|\E\underline{s}_n(z)-\underline{s}(z)|\to 0,\\
    \sup_{n;z\in\CC_n}\left\|\left(\frac{1}{y_n}\E\underline{s}_n(z)\bI_n+\bI_n\right)^{-1}\right\|<\infty,\\
    \sup_{z\in\CC_n}\left|\frac{\E\underline{s}_n^2(z)}{\left(1+\frac{1}{y_n}\E\underline{s}_n(z)\right)^2}\right|<\xi<1,\\
    \E\left|\tr\D^{-1}(z)\M-\E\tr\D^{-1}(z)\M\right|^2\leq C\|\M\|^2,
\end{gather}
where $\M$ is a non-random $n\times n$ matrix.

Next, we decompose $M_n^{(2)}(z)$ further into two parts as 
\begin{align*}
    M_n^{(2)}(z)=p\left(\E\underline{s}_n(z)-\underline{s}_n^{(1)}(z)\right)+p\left(\underline{s}_n^{(1)}(z)-\underline{s}_n^{(0)}(z)\right),
\end{align*}
where $\underline{s}_n^{(1)}(z)\in\mC^+$ is the unique solution to the following equation 
\begin{align}
    z=-\frac{1}{\underline{s}_n^{(1)}(z)}+\frac{1}{y_n}\int\frac{t}{1+t\underline{s}_n^{(1)}(z)}dF^{\bSig}(t).\label{form:m^{1}(z)_form1}
\end{align}
It is noted that \eqref{form:m^{1}(z)_form1} is a particular case of generalized MP equation formulated as 
\begin{align*}
    z=-\frac{1}{\underline{s}}+y_0\int\frac{t}{1+t\underline{s}}dH(t).
\end{align*}
By the fact that $\CC$ lies outside \eqref{form:support}, one can verify that
\begin{align*}
    \sup_{z\in\CC}\left|\underline{s}_n^{(1)}(z)-\underline{s}(z)\right|\to0,\quad\sup_{z\in\CC}\left|\underline{s}_n^{(0)}(z)-\underline{s}(z)\right|\to0.
\end{align*}
Throughout the rest proof, all bounds and convergence statements hold uniformly in $z\in\CC_n$, so we omit the argument $z$ for simplicity of writing.

On the one hand, since $\bSig$ has one eigenvalue of $0$ and $n-1$ of $n/(n-1)$, \eqref{form:m^{1}(z)_form1} can be expressed as 
\begin{align}
    z=-\frac{1}{\underline{s}_n^{(1)}}+\frac{\frac{1}{y_n}}{1+\frac{n}{n-1}\underline{s}_n^{(1)}}.\label{form:m^1(z)_form2}
\end{align}
Considering
\begin{align*}
    \frac{\underline{s}_n^{(1)}-\underline{s}_n^{(0)}}{\underline{s}_n^{(0)}\underline{s}_n^{(1)}}=\frac{1}{\underline{s}_n^{(0)}}-\frac{1}{\underline{s}_n^{(1)}}=\frac{\frac{1}{y_n}}{1+\underline{s}_n^{(0)}}-\frac{\frac{1}{y_n}}{1+\frac{n}{n-1}\underline{s}_n^{(1)}},
\end{align*}
we have 
\begin{align*}
    \underline{s}_n^{(1)}-\underline{s}_n^{(0)}=&\frac{\frac{1}{y_n}\underline{s}_n^{(0)}\underline{s}_n^{(1)}\left[\frac{n}{n-1}\underline{s}_n^{(1)}-\underline{s}_n^{(0)}\right]}{\left(1+\underline{s}_n^{(0)}\right)\left(1+\frac{n}{n-1}\underline{s}_n^{(1)}\right)}\\
    =&\frac{\frac{1}{y_n}\underline{s}_n^{(0)}\underline{s}_n^{(1)}\left[\underline{s}_n^{(1)}-\underline{s}_n^{(0)}\right]}{\left(1+\underline{s}_n^{(0)}\right)\left(1+\frac{n}{n-1}\underline{s}_n^{(1)}\right)}+\frac{1}{n-1}\cdot\frac{\frac{1}{y_n}\underline{s}_n^{(0)}\left(\underline{s}_n^{(1)}\right)^2}{\left(1+\underline{s}_n^{(0)}\right)\left(1+\frac{n}{n-1}\underline{s}_n^{(1)}\right)},
\end{align*}
which implies 
\begin{align}
    p\left(\underline{s}_n^{(1)}-\underline{s}_n^{(0)}\right)=&\frac{\frac{p}{n-1}\frac{1}{y_n}\underline{s}_n^{(0)}\left(\underline{s}_n^{(1)}\right)^2}{\left(1+\underline{s}_n^{(0)}\right)\left(1+\frac{n}{n-1}\underline{s}_n^{(1)}\right)}\left(1-\frac{\frac{1}{y_n}\underline{s}_n^{(0)}\underline{s}_n^{(1)}}{\left(1+\underline{s}_n^{(0)}\right)\left(1+\frac{n}{n-1}\underline{s}_n^{(1)}\right)}\right)^{-1}\nonumber\\
    \to&\frac{\underline{s}^3}{\left(1+\underline{s}\right)^2-y_0\underline{s}^2}.\label{form:expect_one_hand}
\end{align}

On the other hand, we consider 
\begin{align}
    p\left(\E\underline{s}_n-\underline{s}_n^{(1)}\right)=-\left(1-\frac{1}{y_n}\int\frac{t^2\underline{s}_n^{(1)}\E\underline{s}_n}{\left(1+t\underline{s}_n^{(1)}\right)\left(1+t\E\underline{s}_n\right)}dF^{\bSig}(t)\right)^{-1}p\underline{s}_n^{(1)}\E\underline{s}_nR_n,\label{form:expect_other_hand}
\end{align}
where
\begin{align*}
    R_n=\frac{1}{\E\underline{s}_n}+z-\frac{1}{y_n}\int\frac{t}{1+t\E\underline{s}_n}dF^{\bSig}(t).
\end{align*}
Therefore, it suffices to analyze the limit of $n\underline{s}_n^{(1)}\E\underline{s}_nR_n$. Denote
\begin{align*}
    \K(z)=\E\underline{s}_n(z)\bSig+\bI_n.
\end{align*}
We have 
\begin{align*}
    p\E\underline{s}_nR_n=&p\E\underline{s}_n\left(\frac{1}{\E\underline{s}_n}+z-\frac{1}{y_n}\int\frac{t}{1+t\E\underline{s}_n}dF^{\bSig}(t)\right)\\
    =&p\E\s_1\trans\D^{-1}\K^{-1}\s_1+z\E\underline{s}_n\E\tr\D^{-1}\bSig\K^{-1}
\end{align*}
Applying \eqref{form:leave_one_out} and \eqref{form:relation_to_zm(z)}, $\s_1\trans\D^{-1}=\beta_1\s_1\trans\D_1^{-1}$ and $z\E\underline{s}_n=-\E\beta_1$, which implies 
\begin{align*}
    p\E\underline{s}_nR_n=&p\E\beta_1\s_1\trans\D_1^{-1}\K^{-1}\s_1-\E\beta_1\E\tr\D^{-1}\bSig\K^{-1}\\
    =&p\E\beta_1\s_1\trans\D_1^{-1}\K^{-1}\s_1-\E\beta_1\E\tr\D_1^{-1}\bSig\K^{-1}+\E\beta_1\E\tr\left(\D_1^{-1}-\D^{-1}\right)\bSig\K^{-1}
\end{align*}
Considering respectively the following two terms,
\begin{gather}
    p\E\beta_1\s_1\trans\D_1^{-1}\K^{-1}\s_1-\E\beta_1\E\tr\D_1^{-1}\bSig\K^{-1},\label{form:expect_1}\\
    \E\beta_1\E\tr\left(\D_1^{-1}-\D^{-1}\right)\bSig\K^{-1}.\label{form:expect_2}
\end{gather}
By identity $\beta_1=b_n-b_n^2\gamma_1+\beta_1b_n^2\gamma_1^2$, \eqref{form:expect_1} can be split into three parts. 
For the first part,
\begin{align*}
    pb_n\E\s_1\trans\D_1^{-1}\K^{-1}\s_1-b_n\E\tr\D_1^{-1}\bSig\K^{-1}=0.
\end{align*}
For the second part, by \eqref{form:higher_moment_uniform_control} and Lemma \ref{lem:second_moment},
\begin{align}
    &-pb_n^2\E\gamma_1\s_1\trans\D_1^{-1}\K^{-1}\s_1+b_n^2\E\gamma_1\E\tr\D_1^{-1}\bSig\K^{-1}\nonumber\\
    =&-pb_n^2\E\left(\s_1\trans\D_1^{-1}\s_1-\frac{1}{p}\tr\bSig\D_1^{-1}\right)\left(\s_1\trans\D_1^{-1}\K^{-1}\s_1-\frac{1}{p}\tr\bSig\D_1^{-1}\K^{-1}\right)+o(1)\nonumber\\
    =&-z^2\underline{s}^2\E\left(\frac{2}{p}\tr\D_1^{-2}\K^{-1}-\frac{6}{5p}\tr\left(\D_1^{-1}\circ\D_1^{-1}\K^{-1}\right)-\frac{4}{5np}\tr\D_1^{-1}\tr\D_1^{-1}\K^{-1}\right)+o(1).\label{form:expect_1_limit}
\end{align}
For the third part, by \eqref{form:higher_moment_uniform_control},
\begin{align*}
    &\left|pb_n^2\E\beta_1\gamma_1^2\s_1\trans\D_1^{-1}\K^{-1}\s_1-b_n^2\E\beta_1\gamma_1^2\E\tr\D_1^{-1}\bSig\K^{-1}\right|\\
    =&pb_n^2\left|\cov\left(\beta_1\gamma_1^2,\s_1\trans\D_1^{-1}\K^{-1}\s_1\right)\right|\\
    \leq&pb_n^2\sqrt{\E\left|\beta_1\gamma_1^2\right|^2}\sqrt{\var\left(\s_1\trans\D_1^{-1}\K^{-1}\s_1\right)}\\
    \to&0.
\end{align*}
\eqref{form:expect_2} can be expressed by \eqref{form:leave_one_out} as 
\begin{align}
    \E\beta_1\E\tr\left(\D_1^{-1}-\D^{-1}\right)\bSig\K^{-1}
    =\frac{z^2\underline{s}^2}{p}\E\tr\bSig\D_1^{-1}\bSig\K^{-1}\D_1^{-1}+o(1),\label{form:expect_2_limit}
\end{align}
Collecting \eqref{form:expect_1_limit} and \eqref{form:expect_2_limit}, by \eqref{form:leave_one_out} and \eqref{form:higher_moment_uniform_control}, we obtain
\begin{align*}
    p\E\underline{s}_nR_n=-J_1+\frac{6}{5}J_2+\frac{4}{5}J_3+o(1),
\end{align*}
where 
\begin{align*}
    J_1=&\frac{z^2\underline{s}^2}{p}\left(2\E\tr\D^{-2}\K^{-1}-\E\tr\bSig\D^{-1}\bSig\K^{-1}\D^{-1}\right),\\
    J_2=&\frac{z^2\underline{s}^2}{p}\E\tr\left(\D^{-1}\circ\D^{-1}\K^{-1}\right),\\
    J_3=&\frac{z^2\underline{s}^2}{np}\E\tr\D^{-1}\tr\D^{-1}\K^{-1}.
\end{align*}
As has been discussed in \textbf{Step 1}, multiplying $\bSig$ has no influence on $J_1$, which implies 
\begin{align*}
    J_1=\frac{z^2\underline{s}^2}{p}\E\tr\bSig\D^{-1}\bSig\K^{-1}\D^{-1}+o(1).
\end{align*}
Following the same procedures on Pages 589-592 of \cite{bai2004clt} and applying \eqref{form:higher_moment_uniform_control}, we have 
\begin{align}
    J_1=\frac{y_0\underline{s}^2}{\left(\left(1+\underline{s}\right)^2-y_0\underline{s}^2\right)\left(1+\underline{s}\right)}+o(1).\label{form:J_1_simplified}
\end{align}
With similar arguments in \eqref{form:start}-\eqref{form:I_3_simplified} and applying \eqref{form:higher_moment_uniform_control}, we have 
\begin{align}
    J_2=&\frac{z^2\underline{s}^2}{p}\tr\left(\E\D^{-1}\circ\E\D^{-1}\K^{-1}\right)+o(1)\nonumber\\
    =&\frac{z^2\underline{s}^2}{p}\tr\left((-z\underline{s}\bSig-z\bI_n)^{-1}\circ(-z\underline{s}\bSig-z\bI_n)^{-1}(\underline{s}\bSig+\bI_n)^{-1}\right)+o(1)\nonumber\\
    =&\frac{y_0\underline{s}^2}{\left(1+\underline{s}\right)^3}+o(1).\label{form:J_2_simplified}
\end{align}
and 
\begin{align}
    J_3=&\frac{z^2\underline{s}^2}{np}\E\tr\E\D^{-1}\tr\E\D^{-1}\K^{-1}+o(1)\nonumber\\
    =&\frac{z^2\underline{s}^2}{np}\tr\left(-z\underline{s}\bSig-z\bI_n\right)^{-1}\tr\left(-z\underline{s}\bSig-z\bI_n\right)^{-1}\left(\underline{s}\bSig+\bI_n\right)^{-1}+o(1)\nonumber\\
    =&\frac{y_0\underline{s}^2}{\left(1+\underline{s}\right)^3}+o(1).\label{form:J_3_simplified}
\end{align}
Collecting \eqref{form:J_1_simplified}-\eqref{form:J_3_simplified}, we obtain 
\begin{align}
    p\E\underline{s}_nR_n=-\frac{y_0\underline{s}^2}{\left(\left(1+\underline{s}\right)^2-y_0\underline{s}^2\right)\left(1+\underline{s}\right)}+\frac{2y_0\underline{s}^2}{\left(1+\underline{s}\right)^3}+o(1).\label{form:pmr}
\end{align}
Together with \eqref{form:expect_one_hand}, \eqref{form:expect_other_hand} and \eqref{form:pmr}, 
\begin{align*}
    p\left(\E\underline{s}_n(z)-\underline{s}_n^{(0)}(z)\right)\to\mu(z).
\end{align*}

\subsection{Proof of Lemma \ref{lem:clt_Ln(z)}}
The main process is similar with the proof of Lemma \ref{lem:clt_Mn(z)}. First we prove the convergence of $\widehat{L}_n(z)$ in probability for each $z\in\CC$, then we show its tightness on $\CC$, which leads to the convergence of stochastic process.

\textbf{Step 1. Convergence of $\widehat{L}_n(z)$.}

Since that 
\begin{align*}
    \brho_n-\tbrho_n=\frac{3}{n+1}\left(\K_n-\tbrho_n\right),
\end{align*}
we have 
\begin{align*}
    \widehat{L}_n(z)=&\tr\left(\tbrho_n-z\bI_p\right)^{-1}\left(\brho_n-\tbrho_n\right)\left(\brho_n-z\bI_p\right)^{-1}\\
    =&\frac{3}{n+1}\tr\left(\tbrho_n-z\bI_p\right)^{-1}\K_n\left(\brho_n-z\bI_p\right)^{-1}\\
    &-\frac{3}{n+1}\tr\left(\tbrho_n-z\bI_p\right)^{-1}\tbrho_n\left(\brho_n-z\bI_p\right)^{-1}.
\end{align*}
We further expand $\left(\brho_n-z\bI_p\right)^{-1}$,
\begin{align*}
    &\tr\left(\tbrho_n-z\bI_p\right)^{-1}\K_n\left(\brho_n-z\bI_p\right)^{-1}-\tr\left(\tbrho_n-z\bI_p\right)^{-1}\K_n\left(\tbrho_n-z\bI_p\right)^{-1}\\
    =&\frac{3}{n+1}\tr\left(\tbrho_n-z\bI_p\right)^{-1}\K_n\left(\tbrho_n-z\bI_p\right)^{-1}\K_n\left(\brho_n-z\bI_p\right)^{-1}\\
    &-\frac{3}{n+1}\tr\left(\tbrho_n-z\bI_p\right)^{-1}\K_n\left(\tbrho_n-z\bI_p\right)^{-1}\tbrho_n\left(\brho_n-z\bI_p\right)^{-1}.
\end{align*}
By the rigidity of the edge of Kendall's rank  correlation matrix, $\|\K_n\|$ is uniformly bounded almost surely, and then subsequently $\tbrho_n=(n+1)\brho_n/(n-2)-3\K_n/(n-2)$ also has uniformly bounded spectral norm. So we have 
\begin{align*}
    &\frac{3}{n+1}\tr\left(\tbrho_n-z\bI_p\right)^{-1}\K_n\left(\brho_n-z\bI_p\right)^{-1}\\
    =&\frac{3}{n+1}\tr\left(\tbrho_n-z\bI_p\right)^{-1}\K_n\left(\tbrho_n-z\bI_p\right)^{-1}+o_P(1),
\end{align*}
and similarly,
\begin{align*}
    &\frac{3}{n+1}\tr\left(\tbrho_n-z\bI_p\right)^{-1}\tbrho_n\left(\brho_n-z\bI_p\right)^{-1}\\
    =&\frac{3}{n+1}\tr\left(\tbrho_n-z\bI_p\right)^{-1}\tbrho_n\left(\tbrho_n-z\bI_p\right)^{-1}+o_P(1),
\end{align*}
which implies
\begin{align*}
    \widehat{L}_n(z)=&\frac{3}{n+1}\tr\left(\tbrho_n-z\bI_p\right)^{-1}\K_n\left(\tbrho_n-z\bI_p\right)^{-1}\\
    &-\frac{3}{n+1}\tr\left(\tbrho_n-z\bI_p\right)^{-1}\tbrho_n\left(\tbrho_n-z\bI_p\right)^{-1}+o_P(1).
\end{align*}
Denote by $\A_i$ the conditional expectation of $\A_{ij}$ given $\X_i$, $\A_i=\E[\A_{ij}|\X_i]$. By the Hoeffding's decomposition of $\A_{ij}$ illustrated in \cite{bandeira2017marvcenko}, \cite{wu2022limiting} and \cite{li2023eigenvalues},
\begin{align*}
    \A_{ij}=\A_i+\A_j+\bvare_{ij},
\end{align*}
where $\bvare_{ij}$ is uncorrelated with $\A_i$ and $\A_j$.
So we naturally approximate $\tbrho_n$ and $\K_n$ by
\begin{align*}
    \U_n=\frac{3}{n}\sum_{i=1}^{n}\A_i\A_i\trans
\end{align*}
and
\begin{align*}
    \V_n=\frac{2}{n}\sum_{i=1}^{n}\A_i\A_i\trans+\frac{1}{3}\bI_p
\end{align*}
respectively.
The error of this approximation can be well controlled as follows.
\begin{lem}\label{lem:approximation}
    Suppose $\X_1,\cdots,\X_n$ i.i.d. from a poplulation $\X\in\mR^p$, whose entries are independent and absolutely continuous respect to the Lebesgue measure. Then we have 
    \begin{align}
        \E\left\|\tbrho_n-\U_n\right\|_F^2=o(p),\label{form:approx_Un}
    \end{align}
    and 
    \begin{align}
        \E\left\|\K_n-\V_n\right\|_F^2=o(p).\label{form:approx_Vn}
    \end{align}
\end{lem}
To replace $\tbrho_n$ and $\K_n$ with $\U_n$ and $\V_n$, we consider 
\begin{align*}
    &\frac{3}{n+1}\tr\left(\tbrho_n-z\bI_p\right)^{-1}\K_n\left(\tbrho_n-z\bI_p\right)^{-1}-\frac{3}{n+1}\tr\left(\tbrho_n-z\bI_p\right)^{-1}\K_n\left(\U_n-z\bI_p\right)^{-1}\\
    =&\frac{3}{n+1}\tr\left(\tbrho_n-z\bI_p\right)^{-1}\K_n\left(\tbrho_n-z\bI_p\right)^{-1}\left(\U_n-\tbrho_n\right)\left(\U_n-z\bI_p\right)^{-1}.
\end{align*}
By Cauchy's inequality and Lemma \ref{lem:approximation},
\begin{align*}
    &\E\left|\tr\left(\tbrho_n-z\bI_p\right)^{-1}\K_n\left(\tbrho_n-z\bI_p\right)^{-1}\left(\U_n-\tbrho_n\right)\left(\U_n-z\bI_p\right)^{-1}\right|\\
    \leq&\E\left(p\cdot\|\K_n\|\left\|\left(\tbrho_n-z\bI_p\right)^{-1}\right\|^2\left\|\left(\U_n-z\bI_p\right)^{-1}\right\|\tr\left(\U_n-\tbrho_n\right)^2\right)^{\frac{1}{2}}\\
    \lesssim&p^{\frac{1}{2}}\left(\E\tr\left(\U_n-\tbrho_n\right)^2\right)^{\frac{1}{2}}\\
    =&o(p),
\end{align*}
which conclude by Markov's inequality that 
\begin{align*}
    &\frac{3}{n+1}\tr\left(\tbrho_n-z\bI_p\right)^{-1}\K_n\left(\tbrho_n-z\bI_p\right)^{-1}\\
    =&\frac{3}{n+1}\tr\left(\tbrho_n-z\bI_p\right)^{-1}\K_n\left(\U_n-z\bI_p\right)^{-1}+o_P(1).
\end{align*}
Following similar steps above, we obtain 
\begin{align*}
    &\frac{3}{n+1}\tr\left(\tbrho_n-z\bI_p\right)^{-1}\K_n\left(\tbrho_n-z\bI_p\right)^{-1}\\
    =&\frac{3}{n+1}\tr\left(\U_n-z\bI_p\right)^{-1}\V_n\left(\U_n-z\bI_p\right)^{-1}+o_P(1).
\end{align*}
And similarly,
\begin{align*}
    &\frac{3}{n+1}\tr\left(\tbrho_n-z\bI_p\right)^{-1}\tbrho_n\left(\tbrho_n-z\bI_p\right)^{-1}\\
    =&\frac{3}{n+1}\tr\left(\U_n-z\bI_p\right)^{-1}\U_n\left(\U_n-z\bI_p\right)^{-1}+o_P(1).
\end{align*}
Observing that 
\begin{align*}
    \tr\left(\U_n-z\bI_p\right)^{-1}\V_n\left(\U_n-z\bI_p\right)^{-1}=&\tr\left(\U_n-z\bI_p\right)^{-2}\V_n\\
    =&\frac{d}{dz}\tr\left(\U_n-z\bI_p\right)^{-1}\V_n,
\end{align*}
so it suffices to find the limit of 
\begin{align}
    \frac{3}{n+1}\tr\left(\U_n-z\bI_p\right)^{-1}\V_n\label{form:item_1}
\end{align}
and 
\begin{align}
    \frac{3}{n+1}\tr\left(\U_n-z\bI_p\right)^{-1}\U_n.\label{form:item_2}
\end{align}
For simplicity of writing, we denote 
\begin{align*}
    \t_i=\sqrt{\frac{3}{n}}\A_i,\quad\H(z)=\U_n-z\bI_p,\quad\H_j(z)=\U_n-z\bI_p-\t_j\t_j\trans.
\end{align*}
As for the first term,
\begin{align*}
    \eqref{form:item_1}=\frac{2}{n+1}\sum_{j=1}^n\t_j\trans\H^{-1}(z)\t_j+\frac{1}{n+1}\tr\H^{-1}(z).
\end{align*}
By the leave-one-out method,
\begin{align*}
    \H^{-1}(z)=\H_j^{-1}(z)-\frac{\H_j^{-1}(z)\t_j\t_j\trans\H_j^{-1}(z)}{1+\t_j\trans\H_j^{-1}(z)\t_j},
\end{align*}
and then 
\begin{align*}
    \t_j\trans\H^{-1}(z)\t_j=\frac{\t_j\trans\H_j^{-1}(z)\t_j}{1+\t_j\trans\H_j^{-1}(z)\t_j}.
\end{align*}
Since that 
\begin{align*}
    &\E\left|\frac{\t_j\trans\H_j^{-1}(z)\t_j-\frac{1}{n}\tr\H^{-1}(z)}{1+\t_j\trans\H_j^{-1}(z)\t_j}\right|\\
    \lesssim&\E\left|\t_j\trans\H_j^{-1}(z)\t_j-\frac{1}{n}\tr\H_j^{-1}(z)\right|+\frac{1}{n}\E\left|\tr\H_j^{-1}(z)-\tr\H^{-1}(z)\right|\\
    =&O(n^{-\frac{1}{2}}),
\end{align*}
and
\begin{align*}
    &\E\left|\frac{1}{1+\t_j\trans\H_j^{-1}(z)\t_j}-\frac{1}{1+\frac{1}{n}\tr\H^{-1}}\right|\\
    =&\E\left|\frac{\t_j\trans\H_j^{-1}(z)\t_j-\frac{1}{n}\tr\H^{-1}(z)}{\left(1+\t_j\trans\H_j^{-1}(z)\t_j\right)\left(1+\frac{1}{n}\tr\H^{-1}(z)\right)}\right|\\
    \lesssim&\E\left|\t_j\trans\H_j^{-1}(z)\t_j-\frac{1}{n}\tr\H_j^{-1}(z)\right|+\frac{1}{n}\E\left|\tr\H_j^{-1}(z)-\tr\H^{-1}(z)\right|\\
    =&O(n^{-\frac{1}{2}}),
\end{align*}
we have 
\begin{align*}
    \eqref{form:item_1}=\frac{2n}{n+1}\cdot\frac{\frac{1}{n}\tr\H^{-1}(z)}{1+\frac{1}{n}\tr\H^{-1}(z)}+\frac{1}{n+1}\tr\H^{-1}(z)+o_P(1).
\end{align*}
By Theorem 2 of \cite{wu2022limiting}, 
\begin{align*}
    \frac{1}{p}\tr\H^{-1}(z)\to m(z)=\frac{1-y-z+\sqrt{(1+y-z)^2-4y}}{2yz},\quad\text{almost surely}.
\end{align*}
Therefore,
\begin{align*}
    \eqref{form:item_1}\to\frac{2ym(z)}{1+ym(z)}+ym(z),\quad\text{in probability}.
\end{align*}
And similarly,
\begin{align*}
    \eqref{form:item_2}=\frac{3}{n+1}\sum_{j=1}^{n}\t_j\trans\H^{-1}(z)\t_j\to\frac{3ym(z)}{1+ym(z)},\quad\text{in probability}.
\end{align*}
To sum up,
\begin{align*}
    \widehat{L}_n(z)=&\frac{3}{n+1}\tr\left(\U_n-z\bI_p\right)^{-2}\left(\V_n-\U_n\right)+o_P(1)\\
    =&\frac{3}{n+1}\frac{d}{dz}\tr\left(\U_n-z\bI_p\right)^{-1}\left(\V_n-\U_n\right)+o_P(1)\\
    \to&ym'(z)-\frac{ym'(z)}{\left(1+ym(z)\right)^2},\quad\text{in probability}.
\end{align*}
Since 
\begin{align*}
    \underline{s}(z)=\frac{-1+y-z+\sqrt{(1+y-z)^2-4y}}{2z},
\end{align*}
we have
\begin{align*}
    y_0\underline{s}(y_0z)=&\frac{-1+y_0-y_0z+\sqrt{(1+y_0-y_0z)^2-4y_0}}{2y_0z}\\
    =&\frac{-y+1-z+\sqrt{(y+1-z)^2-4y}}{2z}=m(z).
\end{align*}
Therefore,
\begin{align*}
    ym'(z)-\frac{ym'(z)}{\left(1+ym(z)\right)^2}=y_0\underline{s}'(y_0z)-\frac{y_0\underline{s}'(y_0z)}{\left(1+\underline{s}(y_0z)\right)^2}=y_0\widetilde{\mu}(y_0z).
\end{align*}

\textbf{Step 2. Tightness of $\widehat{L}_n(z)$.}

Similar with the step 2 of the proof of Lemma \ref{lem:clt_Mn(z)}, it suffices to verify
\begin{align*}
    \sup_{n;z_1,z_2\in\CC_n}\frac{\E|\widehat{L}_n(z_1)-\widehat{L}_n(z_2)|^2}{|z_1-z_2|^2}<\infty.
\end{align*}
Since 
\begin{align*}
    &\frac{\E|\widehat{L}_n(z_1)-\widehat{L}_n(z_2)|^2}{|z_1-z_2|^2}\\
    =&\E\left|\tr\left(\tbrho_n-z_1\bI_p\right)^{-1}\left(\tbrho_n-z_2\bI_p\right)^{-1}-\tr\left(\brho_n-z_1\bI_p\right)^{-1}\left(\brho_n-z_2\bI_p\right)^{-1}\right|^2\\
    =&\E\left|\sum_{i=1}^{p}\frac{1}{(\widetilde{\lambda}_i-z_1)(\widetilde{\lambda}_i-z_2)}-\sum_{i=1}^{p}\frac{1}{(\lambda_i-z_1)(\lambda_i-z_2)}\right|^2\\
    =&\E\left|\sum_{i=1}^{p}\frac{(\lambda_i+\widetilde{\lambda}_i-z_1-z_2)(\lambda_i-\widetilde{\lambda}_i)}{(\widetilde{\lambda}_i-z_1)(\widetilde{\lambda}_i-z_2)(\lambda_i-z_1)(\lambda_i-z_2)}\right|^2,
\end{align*}
where $\lambda_1\geq\cdots\geq\lambda_p$ and $\widetilde{\lambda}_1\geq\cdots\geq\widetilde{\lambda}_p$ are the eigenvalues of $\brho_n$ and $\tbrho_n$ respectively. By Weyl's inequality,
\begin{align*}
    \left|\lambda_i-\widetilde{\lambda}_i\right|\leq\left\|\brho_n-\tbrho_n\right\|=\frac{3}{n+1}\left\|\K_n+\tbrho_n\right\|.
\end{align*}
By the rigidity of edge of $\K_n$ and $\brho_n$ and the truncation of $\widehat{L}_n(z)$,
\begin{align*}
    \E\left|\sum_{i=1}^{p}\frac{(\lambda_i+\widetilde{\lambda}_i-z_1-z_2)(\lambda_i-\widetilde{\lambda}_i)}{(\widetilde{\lambda}_i-z_1)(\widetilde{\lambda}_i-z_2)(\lambda_i-z_1)(\lambda_i-z_2)}\right|^2\leq C_1p\sum_{i=1}^{p}\E\left|\lambda_i-\widetilde{\lambda}_i\right|^2\leq C_2,
\end{align*}
uniformly in $z_1,z_2\in\CC_n$.

\appendix 
\section{Auxiliary lemmas}
\begin{lem}\label{lem:second_moment}
    For non-random $n\times n$ symmetric matrices $\A$ and $\B$, we have 
    \begin{align*}
    \E \s \trans\A\s=&\tr\bSig\A;\\
    \cov(\s \trans\A\s,\s \trans\B\s)=&2\tr(\A\B)-\frac{6}{5}\tr(\A\circ\B)-\frac{4}{5n}\tr(\A)\tr(\B)+O(1)\|\A\|\|\B\|.
    \end{align*}
\end{lem}
\begin{proof}
Noting
\begin{align*}
    \E \s=\mathbf{0}_n,~\cov(\s)=\bSig=\frac{n}{n-1}\left(\bI_n-\frac{1}{n}\one_n\one_n\trans\right),
\end{align*}
we have $ \E \s \trans\A\s=\tr\bSig\A$. Furthermore, we calculate the covariance of the interaction term $ \cov(s_is_j,s_ks_l)$. 

By Faulhaber's formula, 
\begin{align*}
    \var(s_1^2)=\E (s_1^2-1)^2=\frac{144}{n(n-1)^2(n+1)^2}\sum_{i=1}^n\left(i-\frac{n+1}{2}\right)^4-1=\frac{4}{5}-\frac{12}{5(n^2-1)}.
\end{align*}   
Noting $s_1+\ldots+s_n=0$, we have
\begin{align*}
  0=\cov\left(s_1^2,s_1(s_1+\ldots+s_n)\right)=\var(s_1^2)+(n-1)\cov(s_1^2,s_1s_2),
\end{align*}
which yields 
\begin{align*}
   \cov(s_1^2,s_1s_2)=-\frac{1}{n-1}\var(s_1^2).
\end{align*}
Similarly, since $s_1^2+\ldots+s_n^2=n$,
\begin{align*}
  0=\cov\left(s_1^2,s^2_1+\ldots+s^2_n)\right)=\var(s_1^2)+(n-1)\cov(s_1^2,s^2_2),
\end{align*}
which yields 
\begin{align*}
  \cov(s_1^2,s^2_2)=-\frac{1}{n-1}\var(s_1^2).
\end{align*}
By exploiting this trick, we can get
\begin{align*}
\cov(s_1^2,s_2s_3)=&\frac{2}{(n-1)(n-2)}\var(s_1^2),\\
\var(s_1s_2)=&\frac{n(n-2)}{(n-1)^2}-\frac{1}{n-1}\var(s_1^2),\\
\cov(s_1s_2,s_1s_3)=&-\frac{n}{(n-1)^2}+\frac{2}{(n-1)(n-2)}\var(s_1^2),\\
\cov(s_1s_2,s_3s_4)=&\frac{2n}{(n-1)^2(n-3)}-\frac{6}{(n-1)(n-2)(n-3)}\var(s_1^2).
\end{align*}
Now, we are ready to derive the covariance
\begin{align*}
    &\cov(\s \trans\A\s,\s \trans\B\s)=\cov\left(\sum_{i=1}^n a_{ii}s_i^2+2\sum_{i<j} a_{ij}s_is_j,\sum_{k=1}^n b_{kk}s_k^2+2\sum_{k<l} b_{kl}s_ks_l\right)\\
    =&\sum_{i,k} a_{ii}b_{kk}\cov(s_i^2,s_k^2)+2\sum_{i,k<l} a_{ii}b_{kl}\cov(s_i^2,s_ks_l)\\
    &+2\sum_{i<j,k} a_{ij}b_{kk}\cov(s_is_j,s_k^2)+4\sum_{i<j,k<l}a_{ij}b_{kl}\cov(s_is_j,s_ks_l)\\
=&\var(s_1^2)\sum_{i=1}^n a_{ii}b_{ii}+\cov(s_1^2,s^2_2)\sum_{i,k}^* a_{ii}b_{kk}\\
&+2\cov(s_1^2,s_1s_2)\sum_{k,l}^*a_{kk}b_{kl}
+\cov(s_1^2,s_2s_3)\sum_{i,k,l}^*a_{ii}b_{kl}\\ 
&+2\cov(s_1^2,s_1s_2)\sum_{i,j}^*a_{ij}b_{ii}
+\cov(s_1^2,s_2s_3)\sum_{i,j,k}^*a_{ij}b_{kk} \\ 
&+2\var(s_1s_2) \sum_{i,j}^*a_{ij}b_{ij}+4\cov(s_1s_2,s_1s_3)\sum_{i,j,k}^*a_{ij}b_{jk} +\cov(s_1s_2,s_3s_4)\sum_{i,j,k,l}^*a_{ij}b_{kl} \\
=&\left(\var(s_1^2)-\cov(s_1^2,s^2_2)-4\cov(s_1^2,s_2s_3)+4\cov(s_1^2,s_2s_3)-2\var(s_1s_2)+8\cov(s_1s_2,s_1s_3)\right.\\
&\left.-6\cov(s_1s_2,s_3s_4)\right) \tr(\A \circ \B)+\left(\cov(s_1^2,s^2_2)-2(s_1^2,s_2s_3)+\cov(s_1s_2,s_3s_4)\right)\tr(\A)\tr(\B)\\
&+\left(2\var(s_1s_2)-4\cov(s_1s_2,s_1s_3)+2\cov(s_1s_2,s_3s_4)\right)\tr(\A\B)\\
&+\left(2\cov(s_1^2,s_1s_2)-2\cov(s_1^2,s_2s_3)-4\cov(s_1s_2,s_1s_3)+4\cov(s_1s_2,s_3s_4)\right)\one_n\trans\B\diag(\A)\\
&+\left(2\cov(s_1^2,s_1s_2)-2\cov(s_1^2,s_2s_3)-4\cov(s_1s_2,s_1s_3)+4\cov(s_1s_2,s_3s_4)\right)\one_n\trans\A\diag(\B)\\
&+\left(4\cov(s_1s_2,s_1s_3)-4\cov(s_1s_2,s_3s_4)\right)\one_n\trans\A\B\one_n+\cov(s_1s_2,s_3s_4)\one_n\trans\A\one_n\one_n\trans\B\one_n\\
=&2\tr(\A\B)-\frac{6}{5}\tr(\A\circ\B)-\frac{4}{5n}\tr(\A)\tr(\B)+O(1)\|\A\|\|\B\|,
\end{align*}
where we use the facts
\begin{align*}
 \sum_{i,k}^* a_{ii}b_{kk}=&\sum_{i,j} a_{ii}b_{jj}-\sum_{i=1}^na_{ii}b_{ii}= \tr(\A)\tr(\B)-\tr(\A\circ\B);  \\
 \sum_{k,l}^*a_{kk}b_{kl}=&\sum_{i,j}a_{ii}b_{ij}-\sum_{i=1}^na_{ii}b_{ii}=\one_n \trans \B \diag(\A)-\tr(\A\circ\B);\\
 \sum_{i,k,l}^*a_{ii}b_{kl}=& \sum_{i=1}^n a_{ii}\left(\sum_{k,l}b_{kl}+2b_{ii}-2\sum_{k=1}^n b_{il}-\sum_{k=1}^n b_{kk}  \right)\\
 =&\tr(\A) \one_n\trans\B\one_n+2\tr(\A\circ\B)-2\one_n \trans \B \diag(\A)-\tr(\A)\tr(\B);\\
\sum_{i,j}^*a_{ij}b_{ii}=&\one_n \trans \A \diag(\B)-\tr(\A\circ\B);\\
\sum_{i,j,k}^*a_{ij}b_{kk}=&\tr(\B) \one_n\trans\A\one_n+2\tr(\A\circ\B)-2\one_n \trans \A \diag(\B)-\tr(\A)\tr(\B);\\
\sum_{i,j}^*a_{ij}b_{ij}=&\sum_{i,j}a_{ij}b_{ij}-\sum_{i=1}^na_{ii}b_{ii}=\tr(\A\B)-\tr(\A\circ\B);\\
\sum_{i,j,k}^*a_{ij}b_{jk}=&\sum_{i,k}^* \left( \sum_{j=1}^n a_{ij}b_{jk}-a_{ii}b_{ik}-a_{ik}b_{kk} \right)\\
=&\one_n\trans\A\B\one_n-\tr(\A\B)-\one_n\trans\B\diag(\A)-\one_n\trans\A\diag(\B)+2\tr(\A\circ\B)\\
\sum_{i,j,k,l}^*a_{ij}b_{kl}=&\sum_{i,j,k}^*a_{ij}\left(\sum_{l=1}^n b_{kl}-b_{ki}-b_{kj}-b_{kk}\right)\\
=&\sum_{i,j}^*a_{ij}\left(\sum_{k,l}b_{kl}-\sum_{l=1}^nb_{il}-\sum_{l=1}^nb_{jl}\right)-\sum_{i,j,k}^*a_{ij}\left(b_{ki}+b_{kj}+b_{kk}\right)\\
=&\one_n\trans\A\one_n\one_n\trans\B\one_n-2\one_n\trans\A\B\one_n-\tr(\A)\one_n\trans\B\one_n+2\one_n\trans\B\diag(\A)\\
&-2\left(\one_n\trans\A\B\one_n-\tr(\A\B)-\one_n\trans\B\diag(\A)-\one_n\trans\A\diag(\B)+2\tr(\A\circ\B)\right)\\
&-\left(\tr(\B) \one_n\trans\A\one_n+2\tr(\A\circ\B)-2\one_n \trans \A \diag(\B)-\tr(\A)\tr(\B)\right)\\
=&\one_n\trans\A\one_n\one_n\trans\B\one_n-4\one_n\trans\A\B\one_n+4\one_n\trans\A\diag(\B)+4\one_n\trans\B\diag(\A)\\
&-\tr(\A)\one_n\trans\B\one_n-\tr(\B)\one_n\trans\A\one_n-6\tr(\A\circ\B)+2\tr(\A\B)+\tr(\A)\tr(\B)
\end{align*}
and the bounds 
\begin{gather*}
 \tr(\A \circ \B)=O(n)\|\A\|\|\B\|, \tr(\A)\tr(\B)=O(n^2)\|\A\|\|\B\|, \tr(\A\B)= O(n)\|\A\|\|\B\|,\\
 \one_n\trans\B\diag(\A)=O(n)\|\A\|\|\B\|, \one_n\trans\A\diag(\B)=O(n)\|\A\|\|\B\|,\\  \one_n\trans\A\one_n\one_n\trans\B\one_n=O(n^2)\|\A\|\|\B\|,  \one_n\trans\A\B\one_n=O(n)\|\A\|\|\B\|.
\end{gather*}
The proof is completed.
\end{proof}
\begin{lem}\label{lem:higher_moments_control}
    For non-random $n\times n$ matrices $\A_k$, $k=1,\cdots,r$ and $\B_l$, $l=1,\cdots,q$, we have 
    \begin{align}
        \left|\E\prod_{k=1}^r\s_1\trans\A_k\s_1\prod_{l=1}^q\left(\s_1\trans\B_l\s_1-\frac{1}{p}\tr\bSig\B_l\right)\right|\leq Cn^{-\frac{q}{2}+\delta}\prod_{k=1}^r\|\A_k\|\prod_{l=1}^q\|\B_l\|,\label{form:high_moments_control}
    \end{align}
    for arbitrarily small $\delta>0$.
\end{lem}
\begin{proof}
    For non-random $n\times n$ matrix $\B$,
    by Proposition 2.1 of \cite{bao2019tracy_spearman},
    \begin{align}
        \E\left|\s_1\trans\B\s_1-\frac{1}{p}\tr\bSig\B\right|^q\leq Cn^{-\frac{q}{2}+\delta}\|\B\|^q,\label{form:high_moments_control_origin}
    \end{align}
    for arbitrarily small $\delta>0$. When $r=0$, $q>1$, \eqref{form:high_moments_control} is a consequence of \eqref{form:high_moments_control_origin} and Holder's inequality. If $r>0$, by induction on $r$ we have
    \begin{align*}
        &\left|\E\prod_{k=1}^r\s_1\trans\A_k\s_1\prod_{l=1}^q\left(\s_1\trans\B_l\s_1-\frac{1}{p}\tr\bSig\B_l\right)\right|\\
        \leq&\left|\E\prod_{k=1}^{r-1}\s_1\trans\A_k\s_1(\s_1\trans\A_r\s_1-\frac{1}{p}\tr\bSig\A_r)\prod_{l=1}^q\left(\s_1\trans\B_l\s_1-\frac{1}{p}\tr\bSig\B_l\right)\right|\\
        &+\frac{n}{p}\|\A_r\|\left|\E\prod_{k=1}^{r-1}\s_1\trans\A_k\s_1\prod_{l=1}^q\left(\s_1\trans\B_l\s_1-\frac{1}{p}\tr\bSig\B_l\right)\right|\\
        \leq&Cn^{-\frac{q}{2}+\delta}\prod_{k=1}^r\|\A_k\|\prod_{l=1}^q\|\B_l\|,
    \end{align*}
    which conclude the result.
\end{proof}

\begin{lem}[Theorem 35.12 of \cite{billingsley1995probability}]\label{lem:martingale_clt}
    Suppose for each $n$, $Y_{n1},\cdots,Y_{nr_n}$ is a real martingale difference sequence with respect to $\sigma$-field $\{\F_{n,j}\}$ having finite second moments. If 
    \begin{align*}
        \sum_{j=1}^{r_n}\E\left(Y_{nj}^2|\F_{n,j-1}\right)\to\sigma^2,\quad\text{in probability},
    \end{align*}
    and for each $\varepsilon>0$,
    \begin{align*}
        \sum_{j=1}^{r_n}\E\left(Y_{n,j}^2I(|Y_{n,j}|\geq\varepsilon)\right)\to0,
    \end{align*}
    then 
    \begin{align*}
        \sum_{j=1}^{r_n}Y_{n,j}\to N(o,\sigma^2),\quad\text{in distribution}.
    \end{align*}
\end{lem}

\begin{proof}[Proof of Theorem \ref{thm:calculation}]
    Let $f(x)=x^k$, the centering term is 
    \begin{align*}
        \int x^kdF_{y_n}(x)=\sum_{j=0}^{k-1}\frac{y_n^j}{j+1}\binom{k}{j}\binom{k-1}{j}.
    \end{align*}
    The asymptotic mean and asymptotic variance can be derived with the same approach in Theorem 1.4 of \cite{pan2008central}, so we omit details here. Then we consider the case of $f(x)=\log(x)$.
    By (3.6) of \cite{bai2009corrections}, the centering term is 
    \begin{align*}
        \int \log(x)dF_{y_n}(x)=\frac{y_n-1}{y_n}\log(1-y_n)-1.
    \end{align*}
    Next, for asymptotic mean and asymptotic variance, we notice that
    \begin{align*}
        \underline{s}'(z)=\frac{\underline{s}^2(z)\left(1+\underline{s}(z)\right)^2}{\left(1+\underline{s}(z)\right)^2-y_0\underline{s}^2(z)}.
    \end{align*}
    Since we consider the case of $y<1$, then $y_0=1/y>1$. when $x>(1+\sqrt{y_0})^2$, we have $0>\underline{s}(x)>-1$, and when $x<(1-\sqrt{y_0})^2$, we have $\underline{s}(x)>1/(y_0-1)$. Then we calculate for $k\geq 2$,
    \begin{align*}
    \int\frac{\log(z)\underline{s}'(z)}{\left(1+\underline{s}\right)^k}dz=&\int\frac{\log(z(\underline{s}))}{(1+\underline{s})^k}d\underline{s}\\
    =&\frac{1}{k-1}\int\frac{(\underline{s}+1)^2-y_0\underline{s}^2}{(\underline{s}+1)^k}\cdot\frac{1}{\underline{s}\left((y_0-1)\underline{s}-1\right)}d\underline{s}\\
    =&\frac{2\pi i}{k-1}\left[1-(\frac{y_0-1}{y_0})^{k-1}\right],
    \end{align*}
    and for $k=1$,
    \begin{align*}
        \int\frac{\log(z)\underline{s}'(z)}{1+\underline{s}}dz=&\int\frac{\log(z(\underline{s}))}{1+\underline{s}}d\underline{s}\\
    =&\int\frac{\left((\underline{s}+1)^2-y_0\underline{s}^2\right)\log(1+\underline{s})}{\underline{s}+1}\cdot\frac{1}{\underline{s}\left((y_0-1)\underline{s}-1\right)}d\underline{s}\\
    =&-2\pi i\log(1-\frac{1}{y_0}).
    \end{align*}
    As for $\mu_{\log}$ and $\widetilde{\mu}_{\log}$, with similar rouine in section 5 of \cite{bai2004clt},
    \begin{align*}
        -\frac{1}{2\pi i}\int\log(yz)\mu_1(z)dz=&\frac{1}{2\pi}\int_{(1-\sqrt{y_0})^2}^{(1+\sqrt{y_0})^2}\frac{1}{x}\arg\left(1-\frac{y_0\underline{s}^2(x)}{\left(1+\underline{s}(x)\right)^2}\right)dx\\
    =&\frac{1}{2\pi}\int_{(1-\sqrt{y_0})^2}^{(1+\sqrt{y_0})^2}\frac{1}{x}\arctan\left(\frac{x-1-y_0}{\sqrt{4y_0-(x-1-y_0)^2}}\right)dx\\
    =&\frac{1}{2}\log(1-\frac{1}{y_0}).
    \end{align*}
    For other terms,
    \begin{align*}
        -\frac{1}{2\pi i}\int\log(yz)\mu_2(z)dz=&\frac{y_0}{\pi i}\int\frac{\log(z)\underline{s}(z)\underline{s}'(z)}{\left(1+\underline{s}(z)\right)^3}dz\\
    =&\frac{y_0}{\pi i}\int\frac{\log(z(\underline{s}))}{(1+\underline{s})^2}d\underline{s}-\frac{y_0}{\pi i}\int\frac{\log(z(\underline{s}))}{(1+\underline{s})^3}d\underline{s}\\
    =&\frac{1}{y_0},\\
        -\frac{1}{2\pi i}\int\log(yz)\mu_3(z)dz=&-\frac{1}{2\pi i}\int\frac{\log(z)\underline{s}(z)\underline{s}'(z)}{\left(1+\underline{s}(z)\right)^2}dz\\
        =&-\frac{1}{2\pi i}\int\frac{\log(z(\underline{s}))}{1+\underline{s}}d\underline{s}+\frac{1}{2\pi i}\int\frac{\log(z(\underline{s}))}{(1+\underline{s})^2}d\underline{s}\\
        =&\log(1-\frac{1}{y_0})+\frac{1}{y_0},\\
        -\frac{1}{2\pi i}\int\log(yz)\widetilde{\mu}(z)dz=&-\frac{1}{2\pi i}\int\frac{\log(z)\underline{s}(z)\underline{s}'(z)}{1+\underline{s}(z)}dz-\frac{1}{2\pi i}\int\frac{\log(z)\underline{s}(z)\underline{s}'(z)}{\left(1+\underline{s}(z)\right)^2}dz\\
        =&-\frac{1}{2\pi i}\int\log(z(\underline{s}))d\underline{s}+\frac{1}{2\pi i}\int\frac{\log(z(\underline{s}))}{(1+\underline{s})^2}d\underline{s}\\
        =&\frac{1}{y_0-1}+\frac{1}{y_0}.
    \end{align*}
    As for $\sigma_{\log}$ and $\widetilde{\sigma}_{\log}$,
    \begin{align*}
        &-\frac{1}{4\pi^2}\iint\log(yz_1)\log(yz_2)\sigma_1(z_1,z_2)dz_1dz_2\\
        =&-\frac{1}{2\pi^2}\int\log(z_2)\underline{s}'(z_2)\int\frac{\log(z_1)\underline{s}'(z_1)}{\left(\underline{s}(z_1)-\underline{s}(z_2)\right)^2}dz_1dz_2\\
        =&-\frac{1}{\pi i}\int\log(z(\underline{s}_2))\left(\frac{1}{\underline{s}_2}-\frac{1}{\underline{s}_2-\frac{1}{y_0-1}}\right)d\underline{s}_2\\
        =&-\frac{1}{\pi i}\int\log(\frac{\underline{s}_2-\frac{1}{y_0-1}}{\underline{s}_2})\left(\frac{1}{\underline{s}_2}-\frac{1}{\underline{s}_2-\frac{1}{y_0-1}}\right)d\underline{s}_2\\
        &+\frac{1}{\pi i}\int\log(\underline{s}_2+1)\left(\frac{1}{\underline{s}_2}-\frac{1}{\underline{s}_2-\frac{1}{y_0-1}}\right)d\underline{s}_2\\
        =&-2\log(1-\frac{1}{y_0}),
    \end{align*}
    and 
    \begin{align*}
        &-\frac{1}{4\pi^2}\iint\log(yz_1)\log(yz_2)\sigma_2(z_1,z_2)dz_1dz_2\\
        =&\frac{y_0}{2\pi^2}\int\frac{\log(yz_1)\underline{s}'(z_1)}{\left(1+\underline{s}(z_1)\right)^2}dz_1\int\frac{\log(yz_2)\underline{s}'(z_2)}{\left(1+\underline{s}(z_2)\right)^2}dz_2\\
        =&-\frac{2}{y_0}.
    \end{align*}
    Collecting all the above terms, we complete our calculations.
\end{proof}

\begin{proof}[Proof of Lemma \ref{lem:approximation}]
    For population $\X=(X_1,\cdots,X_p)\trans$, we write $F_i$ as the distribution function of $X_i$. Since $X_i$ is absolutely continuous respect to the Lebesgue measure, $F_i(X_i)$ is uniformly distributed on $[0,1]$ and $Y_i=\Phi^{-1}(F_i(X_i))$ is a standard Gaussian distribution. Rank statistics are invariant under this monotonic transformation, that is $r(X_i)=r(Y_i)$ for $i=1,\cdots,n$. Therefore, \eqref{form:approx_Un} and \eqref{form:approx_Vn} are obtained in \cite{wu2022limiting} and \cite{li2023eigenvalues}. 
\end{proof}

\bibliography{ref}
\end{document}